\documentclass[11pt]{article}
\usepackage{amsmath,amsthm,amssymb,enumerate,mathscinet}
\usepackage{fullpage}
\usepackage{xcolor}

\newtheorem{theorem}{Theorem}[section]
\newtheorem{fact}[theorem]{Fact}
\newtheorem{lemma}[theorem]{Lemma}
\newtheorem{question}[theorem]{Question}

\newtheorem{claim}[theorem]{Claim}

\newtheorem{thm}[theorem]{Theorem}
\newtheorem{prop}[theorem]{Proposition}

\newtheorem{conj}[theorem]{Conjecture}
\newtheorem{constru}[theorem]{Construction}
\newtheorem{observation}[theorem]{Observation}

\newcommand*{\mybox}[2]{\colorbox{#1!30}{\parbox{.98\linewidth}{#2}}}

\numberwithin{equation}{section}
\def\A{\mathcal{A}}
\def\B{\mathcal{B}}
\def\P{\mathcal{P}}
\def\eps{\varepsilon}

\def\F{\mathcal{F}}

\def\T{\mathcal{T}}

\def\mis{\text{mis}}

\def\COMMENT#1{}

\title{Applications of graph containers in the Boolean lattice} 

\author{J\'ozsef Balogh,\footnote{Department of Mathematical Sciences,
 University of Illinois at Urbana-Champaign, Urbana, Illinois 61801, USA {\tt
jobal@math.uiuc.edu}. Research is partially supported by Simons Fellowship, NSA Grant H98230-15-1-0002, NSF
 Grant DMS-1500121, Arnold O. Beckman Research Award (UIUC Campus Research Board 15006).}
\   Andrew Treglown\footnote{University of Birmingham, United Kingdom, {\tt a.c.treglown@bham.ac.uk}. Research support by EPSRC grant EP/M016641/1.}
\ and Adam Zsolt Wagner\footnote{University of Illinois at Urbana-Champaign, Urbana, Illinois 61801, USA {\tt
zawagne2@illinois.edu}. }}

\begin{document}
 \maketitle
\begin{abstract}
We apply the graph container method to prove a number of counting results for the Boolean lattice $\mathcal P(n)$. In particular, we: 
\begin{itemize}
\item[(i)] Give a partial answer to a question of Sapozhenko estimating the number of $t$ error correcting codes  in $\mathcal P(n)$, and we also give an upper bound on the number of transportation codes;
\item[(ii)] Provide an alternative proof of Kleitman's theorem on the number of antichains in $\mathcal P(n)$ and  give a two-coloured analogue;
\item[(iii)] Give an asymptotic formula for the number of $(p,q)$-tilted Sperner families in $\mathcal P(n)$;
\item[(iv)] Prove a random version of Katona's $t$-intersection theorem.
\end{itemize}
In each case, to apply the container method, we first prove corresponding supersaturation results.
We also give a construction which disproves two conjectures of Ilinca and Kahn on maximal independent sets and antichains in the Boolean lattice.
 A number of open questions are also given.
\end{abstract}
\COMMENT{AT: all logs base 2? - AZW: think so.}
\section{Introduction}
Many problems in combinatorics and other areas can be rephrased into questions about independent sets in (hyper)graphs.
For example, Sperner's theorem~\cite{sperner} states that the largest antichain in the power set of $[n]$, $\mathcal P(n)$ has size $\binom{n}{\lfloor  n/2\rfloor}$. ($\mathcal P(n)$ is also refered to as the \emph{Boolean lattice}.)
Let $G$ be the graph with vertex set $\mathcal P(n)$ and where $A$ and $B$ are adjacent if $A\subset B$ or $B \subset A$. Then equivalently, Sperner's theorem states that the largest independent set in $G$ has size $\binom{n}{\lfloor  n/2\rfloor}$.

So-called \emph{container results} have  emerged as powerful tools for attacking problems which reduce to counting independent sets in (hyper)graphs. Roughly speaking, container results typically state that the independent sets of a given (hyper)graph $H$ lie only in a `small' number of subsets of the vertex set of $H$ (referred to as \emph{containers}), where each of these containers is an `almost independent set'. 
The \emph{graph} container method dates back to work of Kleitman and Winston~\cite{kw1, kw2} from more than 30 years ago. Indeed, they constructed a relatively simple algorithm that can be used to produce graph container results. 
This algorithm will be the starting point for proving the container results of this paper; we give a more detailed overview of the method in Section~\ref{sec:over}. 
An excellent recent survey of Samotij~\cite{samsurvey} gives several applications of this method to a range of problems in combinatorics and number theory.

The container method has also been  recently generalised to hypergraphs of higher uniformity. Perhaps the first applications of the hypergraph container method appeared in~\cite{bs2}.
Balogh, Morris and Samotij~\cite{container1} and independently Saxton and Thomason~\cite{container2} developed  general container theorems for hypergraphs whose edge distribution satisfies  certain boundedness conditions. 
These results have been used to tackle a range of  important problems including questions arising in combinatorial number theory, Ramsey theory,  positional games, list colourings of graphs and $H$-free graphs.

In this paper we provide several new short applications of the graph container method to counting problems in the Boolean lattice. 
In Section~\ref{sec:tilt} we asymptotically determine
the number of $(p,q)$-tilted Sperner families in $\mathcal P(n)$. 
In Section~\ref{sec:code} we give an upper bound on the number of $t$ error correcting codes, thereby giving a partial answer to a question of Sapozhenko \cite{saponew},  and an upper bound on
the number of so-called $2$-$(n,k,d)$-codes. Katona's intersection theorem~\cite{katona} determines the largest $t$-intersecting family in $\mathcal P(n)$. In Section~\ref{sec:int} we prove a random analogue of this result. We also prove counting versions of 
generalisations of Sperner's theorem: we give an alternative proof of a famous result of Kleitman~\cite{kleitmansperner} that gives an asymptotic formula for the number of antichains in $\mathcal P(n)$ (see Section~\ref{seca}). 
We then prove a two-coloured generalisation of this result in Section~\ref{secb}.
Some open problems are raised in Section~\ref{sec73}. Finally, in Section~\ref{newestsec} we give a construction which disproves two conjectures of Ilinca and Kahn~\cite{IK} on maximal independent sets and antichains in the Boolean lattice.
\COMMENT{AT: tweaked this paragraph slightly}

Section~\ref{sec:over} describes the general algorithm used for producing our graph container results. After this, each of the sections are self-contained and so can be read separately. 
However, there are two important themes which run throughout the paper and which we are keen to publicise. Firstly, for the proof of each of our container theorems, the key step is to apply various \emph{supersaturation} results. Roughly speaking, such results
state that if a vertex set $S$ in some auxiliary graph $G$ is significantly bigger than the size of the largest independent set, then $G[S]$ contains many edges.
Secondly, in some cases we need to apply a \emph{multi-stage} version of the Kleitman--Winston algorithm (and apply more than one supersaturation result). We explain this in more detail in Section~\ref{sec:over}.

\section{Notation and preliminaries}

For a given $n \in \mathbb N$, write $[n]:=\{1, \dots, n\}$. 
Denote $S_n$ the set of all permutations of $[n]$. Given a set $X$ we write $\mathcal P(X)$ for the set of all subsets of $X$.
Given  $k \in \mathbb N$, we write $\binom{X}{\leq k}$ to denote the set of all subsets of $X$ of size at most $k$ and define $\binom{X}{ k}$ and $\binom{X}{\geq k}$ analogously.
Given  $n\in \mathbb N$, we write, for example, $\binom{n}{\geq k}:=\binom{n}{ k}+\binom{n}{ k+1}+\dots+ \binom{n}{ n}$.
We say two sets $A$, $B$ are \emph{comparable} if $A \subset B$ or $B \subset A$. 

Given a graph $G$ we write $N_G(x)$ for the neighbourhood of a vertex $x \in V(G)$ and set $\deg _G (x):=|N_G (x)|$. We write $\Delta(G)$ for the \emph{maximum degree of $G$}.

Throughout the paper we omit floors and ceilings where the argument is unaffected. We write $0<\alpha \ll \beta \ll \gamma$ to mean that we can choose the constants
$\alpha, \beta, \gamma$ from right to left. More
precisely, there are increasing functions $f$ and $g$ such that, given
$\gamma$, whenever we choose  $\beta \leq f(\gamma)$ and $\alpha \leq g(\beta)$, all
calculations needed in our proof are valid. 
Hierarchies of other lengths are defined in the obvious way.


The following  well known bounds for  binomial coefficients will be useful later on.
\begin{fact}\label{fact1}
$$\binom{n}{n/2} \sim \sqrt{\frac{2}{\pi n}} 2^n.$$
\end{fact}
\begin{fact}\label{fact2} If $k=(n+c\sqrt{n})/2$ where $c=o(n^{1/6})$ then
$$\binom{n}{k} \sim \binom{n}{n/2} e^{-(c^2/2)}.$$
\end{fact}
\begin{fact}\label{fact3} For any $n,k \in \mathbb N$,
$$\binom{n}{k} \leq \left  ( \frac{e \cdot n}{k} \right  ) ^k.$$
\end{fact}

\COMMENT{define$\sim$}

\section{The graph container algorithm}\label{sec:over}
For each of our problems, we will prove and then apply a \emph{container} result. We will first introduce some auxiliary graph $G$. For example, to prove Kleitman's theorem on antichains in the Boolean lattice, we will
define $G$ to have vertex set $\mathcal P(n)$ where distinct $A$ and $B$ are adjacent if they are comparable. Most of our container results then take the following general structure:
Let $I_{\max}$ denote the size of the largest independent set in $G$. Then there is a collection $\mathcal F$ of subsets of $V(G)$ such that:
\begin{itemize}
\item[(i)] $|\mathcal F|=2^{o(|I_{\max}|)}$;
\item[(ii)]  Every independent set $I$ in $G$ lies in some $F \in \mathcal F$;
\item[(iii)] $|F| \leq (1+o(1))|I_{\max}|$ for each $F \in \mathcal F$.
\end{itemize}
We refer to the elements of $\mathcal F$ as \emph{containers}. In some cases, when we only have an upper bound $ D$ on  $|I_{\max}|$, we in fact have $D$ instead of $|I_{\max}|$ in (i) and (iii).
Typically the container result will then immediately imply our desired counting theorem. For example, in the case of Kleitman's theorem, since independent sets in $G$ correspond to antichains in $\mathcal P(n)$,  we have that $|I_{\max}|=\binom{n}{\lfloor n/2 \rfloor}$. Thus, (i)--(iii) imply that there are $2^{(1+o(1))\binom{n}{\lfloor n/2 \rfloor}}$ antichains in $\mathcal P(n)$, as desired.

To prove each of our container results we will apply the following algorithm of Kleitman and Winston~\cite{kw1, kw2}.

\smallskip

{\noindent \bf The graph container algorithm.} Let  $V:=V(G)$, $n:=|V|$ and fix an arbitrary total order $v_1,\dots , v_n$ of $V$ and some $\Delta>0$. Let $I$ be an independent set in $G$. Set $G_0 :=G$ and $S:=\emptyset$. In Step~$i$ of the algorithm we do the following:
\begin{itemize}
\item[(a)] Let $u$ be the vertex of maximum degree in $G_{i-1}$ (ties are broken here by our fixed total ordering);
\item[(b)] If $u \notin I$ then define $G_i:= G_{i-1}\setminus \{u\}$ and move to Step~$i+1$;
\item[(c)] If $u \in I$ and $\deg _{G_{i-1} }(u) \geq \Delta $ then add $u$ to $S$; define $G_i:= G_{i-1}\setminus( \{u\}\cup N_G (u))$ and move to Step~$i+1$;
\item[(d)] If $u \in I$ and $\deg _{G_{i-1}} (u) < \Delta $ then define  $f(S):=V(G_i)$ and terminate.
\end{itemize}
Note that $I \subseteq S \cup f(S)$. We sometimes refer to $\Delta$ as the \emph{parameter} of the algorithm.

The algorithm produces a function $f:\binom{V}{\leq |V|/\Delta} \rightarrow \mathcal P(V)$. 
Indeed, the algorithm ensures that $|S|\leq |V|/\Delta$ and that $f$ is  well-defined. 

Let $\mathcal F$ denote the collection of sets $S\cup f(S)$ for each $S \in \binom{V}{\leq |V|/\Delta}$. By construction (ii) is satisfied. There are $\binom{V}{\leq |V|/\Delta}$ containers in $\mathcal F$. Thus, if one chooses $\Delta$ sufficiently large we can ensure that (i) is satisfied. At the end of the algorithm, $G_i$ has maximum degree less than $\Delta$, so is `sparse'. 
In a standard application of the algorithm,
 we then apply a supersaturation result to ensure that (iii) holds: roughly speaking, since $G_i$ is sparse it cannot be too much bigger than the largest independent set in $G$. Hence, $G_i$ and so $S \cup f(S)$ is not too big.

In some cases though, the value of $\Delta$ required to ensure that (i) holds is not small enough to immediately ensure (iii) also holds. That is, $\Delta(G_i)\leq \Delta$ may not imply that (iii) holds.
In this case we have to analyse the algorithm more carefully. Roughly speaking, the idea is to first apply the algorithm with some relatively large parameter $\Delta '$. This will ensure (i) holds and by applying 
a supersaturation result the graph $G_i$ is not too big (though perhaps much bigger than $(1+o(1))|I_{\max}|$). We then continue the algorithm with a new, much smaller parameter $\Delta$ to ensure at the end of this process $G_i$ is much sparser and so (via another supersaturation result) (iii) is satisfied.
We will use this multi-stage approach in Section~\ref{sec:code}. This idea was first used only very recently in~\cite{sper}\COMMENT{AZW: are you sure about this? I assumed that the multi-stage approach was folklore and known for much longer? ----- Yes I asked Jozsi, you were right.} to prove a random analogue of Sperner's theorem.
We remark that when applying this approach in Section~\ref{sec:code}, we will not explicitly state it in this way (we only state the parameter $\Delta$ explicitly and then split the analysis of the algorithm in two), but the method described above is (implicitly) precisely what is happening.

\section{Tilted Sperner families}\label{sec:tilt}
Let $\P(n)$ denote the power set of $[n]$,  ordered by inclusion. A subset $\A \subseteq \P(n)$ is an \emph{antichain} if for any $A, B \in \A$ with $A \subseteq B$ we have $A = B$. So $\binom{[n]}{k}$ is an antichain for any $0 \leq k \leq n$. A celebrated theorem of  Sperner~\cite{sperner} states that in fact no antichain in $\P(n)$ has size larger than $\binom{n}{\lfloor n/2 \rfloor}$. 

Given $A,B \subseteq [n]$ the \emph{subcube of $\P(n)$ spanned by $A$ and $B$} consists of all subsets of $A\cup B$ that contain $A\cap B$.
Kalai (see~\cite{long}) observed that $\mathcal A$ is an antichain precisely if it does not contain $A$ and $B$ such that, in the subcube of $\P(n)$ spanned by $A$ and $B$, $A$ is the top point and $B$ is the bottom point. He asked what happens if one `tilts' this condition. That is, for some $p,q \in \mathbb N$ we forbid $A$ to be $p/(p+q)$ of the way up this subcube and $B$ to be $q/(p+q)$ of the way up this subcube.
More precisely, we say that $\A \subseteq \P(n)$ is a \emph{$(p,q)$-tilted Sperner family} if $\A$ does not contain distinct $A,B$ such that $q|A\setminus B|=p|B\setminus A|$. So the case when $p\not = 0$, $q=0$ corresponds to antichains.

Let  $p,q\in \mathbb N$ be coprime with $p<q$. Leader and Long~\cite{long} proved that the largest $(p,q)$-tilted Sperner family in $\P(n)$ has size $(q-p+o(1))\binom{n}{n/2}$, where the lower bound is obtained by considering the union of the $q-p$ middle layers of the Boolean lattice (see~\cite{long} for an explanation of this).

In 1897, Dedekind~\cite{dedekindsperner} raised the question of how many antichains there are in $\P(n)$. This was famously resolved asymptotically by Kleitman~\cite{kleitmansperner} who proved that there are in fact 
$2^{(1+o(1))\binom{n}{n/2}}$ antichains. In this section we prove an analogue of this result for $(p,q)$-tilted Sperner families.

\begin{theorem}\label{tilt}
Let $p,q\in \mathbb N$ be coprime with $p<q$. Then there are 
$$2^{(q-p+o(1))\binom{n}{n/2}}$$
$(p,q)$-tilted Sperner families in $\P(n)$.
\end{theorem}
To prove Theorem~\ref{tilt} we will apply the following \emph{supersaturation} version of the Leader--Long theorem~\cite{long}. The proof applies the same  averaging argument strategy used in~\cite{long}.

\begin{lemma}\label{supertilt}
Let $p,q\in \mathbb N$ be coprime with $p<q$. 
Given any $\eps >0$, there exist $\delta >0$ and $n_0 \in \mathbb N$ such that the following holds. Suppose that $n \geq n_0$ and $\mathcal A\subseteq \P(n)$ such that 
$|\mathcal A|\geq (q-p+\eps) \binom{n}{n/2}$. Then there are at least $\delta \binom{n}{n/2} n^{p+q}$ pairs $A,B \in \mathcal A$ such that $q|A\setminus B|=p|B\setminus A|$.
\end{lemma}
We remark that the conclusion of Lemma~\ref{supertilt} is actually somewhat stronger than what is needed in the application to the proof of Theorem~\ref{tilt}.
Indeed, for our application instead of $\delta \binom{n}{n/2} n^{p+q}$  such pairs, having only $\delta n2^n$ would be sufficient.
\proof
Given $\eps >0$, define  $\delta >0$ and $C, n_0 \in \mathbb N$ such that
$$0<1/n_0 < \delta  \ll 1/C \ll \eps , 1/p,1/q.$$
Let $n\geq n_0$ and $\mathcal A\subseteq \P(n)$ such that 
$|\mathcal A|\geq (q-p+\eps) \binom{n}{n/2}$.

Let $\mathcal A_i$ denote the set of $A \in \mathcal A$ with $|A|=i$. Since $1/n_0 \ll 1/C \ll \eps$,
$$\sum _{i \geq n/2 +C \sqrt{n}} \binom{n}{i} +\sum _{i \leq n/2 -C \sqrt{n}} \binom{n}{i} \leq \frac{\eps}{2} \binom{n}{n/2}.$$
Thus, we may assume that $|\mathcal A|\geq (q-p+\eps/2) \binom{n}{n/2}$ and every $A \in \mathcal A$ satisfies
$n/2 -C \sqrt{n} \leq |A|\leq n/2 +C \sqrt{n}$.

For simplicity we may assume that $n=(p+q)m$ for some $m \in \mathbb N$ (the other cases follow identically). Clearly there exists $k \in [0,q-p-1]$ such that
\begin{align} \label{kbound}
\sum _{i \equiv k \mod (q-p)} |\mathcal A_i| \geq \left(1+ \frac{\eps}{2(q-p)}\right) \binom{n}{n/2}.
\end{align}
Define $k' \in [0, q-p-1]$ so that $k' \equiv k-pm \mod (q-p)$.

Pick a random ordering of $[n]$ that we denote by $(a_1,  \ldots, a_{qm},b_1,\ldots,b_{pm})$ (this can be viewed as a permutation of $[n]$). Given this ordering, define $C_i:=\{a_j : j\in [qi+k']\}\cup\{b_{j'} : j'\in [pi+1,pm]\}$ and set $\mathcal{C}:=\{C_i : i\in [0,m-1]\}.$ 
Notice that for every $i<j$ we have $|C_i|=pm+k'+(q-p)i \equiv k \mod (q-p)$ and $p|C_j\backslash C_i|=q|C_i\backslash C_j|$. Further, for each $ i \equiv k \mod (q-p)$ where
$pm+k'\leq n/2 -C \sqrt{n} \leq i\leq n/2 +C \sqrt{n}\leq qm+k'$, there is precisely one set $C_{i'}$ in $\mathcal C$ of size $i$.\COMMENT{The size of $C_i$ range from $pm+k'$ up to $n$. $pm+k'$ is much smaller than $n/2 -C \sqrt{n}$, so all set sizes required are covered in $\mathcal C$}

Consider the random variable $X:=|\mathcal{A}\cap \mathcal{C}|$. 
Let  $i \equiv k \mod (q-p)$ where $n/2 -C \sqrt{n} \leq i\leq n/2 +C \sqrt{n}$, and set $i'$ so that $i=pm+k'+(q-p)i'$.
Note that each set $B\in \binom{[n]}{i}$ is equally likely to be  $C_{i'}$, therefore $\mathbb P[B\in\mathcal{C}]=\frac{1}{\binom{n}{i}}$. So 
\begin{align}\label{exbound}
\mathbb EX=\sum_{i \equiv k \mod (q-p)} \frac{|\mathcal{A}_i|}{\binom{n}{i}}\geq \sum_{i \equiv k \mod (q-p)} \frac{|\mathcal{A}_i|}{\binom{n}{n/2}} \stackrel{(\ref{kbound})}{\geq} 1+\frac{\eps}{2(q-p)}.
\end{align}

Consider any permutation $\pi \in S_n$. Write $\pi$ as $(a'_1,  \ldots, a'_{qm},b'_1,\ldots,b'_{pm})$. Define $C_{\pi,i}:=\{a'_j : j\in [qi+k']\}\cup\{b'_{j'} : j'\in [pi+1,pm]\}$ and set $\mathcal{C_ \pi}:=\{C_{\pi,i} : i\in [0,m-1]\}.$ So the set $\mathcal C$ is simply $\mathcal C_\pi$ for a randomly selected permutation $\pi$. Set $\alpha (\pi):=|\mathcal A \cap \mathcal C_\pi|$.
Thus, 
$$\mathbb EX =\frac{1}{n!} \sum _{\pi \in S_n} \alpha (\pi).$$
Together with (\ref{exbound}) this implies that
\begin{align}\label{final}
\sum _{\pi \in S_n} \binom{\alpha (\pi)}{2} \geq \sum _{\pi \in S_n} (\alpha (\pi)-1) \geq \frac{\eps n!}{2(q-p)}.
\end{align}

We say a pair $A,B \in \mathcal A$ is \emph{good} if there is some permutation $\pi$ such that $A,B \in \mathcal C _\pi$. 
That is, $A=C_{\pi,i}$ and $B=C_{\pi,j}$ for some $i,j$ and $\pi \in S_n$. 
In this case, since $A,B \in \mathcal A$, we have $n/2 -C \sqrt{n}\leq |C_{\pi,i}|,| C_{\pi,j}|\leq n/2 +C \sqrt{n}$.
Further, by definition of $\mathcal C_\pi$:
\begin{itemize}
\item[(i)] $|A \cap B| \geq n/2 -(2p+1)C \sqrt{n};$
\item[(ii)]  $q|A\setminus B|=p|B\setminus A|$ or $p|A\setminus B|=q|B\setminus A|$.
\end{itemize}
(By relabeling $A,B$ we may assume that $q|A\setminus B|=p|B\setminus A|$.)\COMMENT{relabeling used implicitly later on}
Moreover, if $A,B$ is good, the definition of the $\mathcal C_\pi$ implies that there are precisely
\begin{align}\label{upper}
|A\cap B|!|A\setminus B|!|B\setminus A|!|\overline{ A\cup B}|!
\end{align}
permutations $\pi$ such that $A,B \in \mathcal C _\pi$. 
Additionally, the following conditions hold:
\begin{itemize}
\item $|A \cap B|, |\overline{A\cup B}| \leq n/2+C\sqrt{n}$;
\item $p\leq |A\setminus B| \stackrel{(i)}{\leq} (2p+2)C \sqrt{n}$;
\item $q \leq |B\setminus A|\stackrel{(i)}{\leq} (2p+2)C \sqrt{n} $.
\end{itemize}
Under these constraints, an upper bound on (\ref{upper}) is
$$p!q!(n/2+C\sqrt{n})!(n/2-C\sqrt{n}-p-q)!.$$
Together with (\ref{final}) this implies that there are at least
$$\frac{\eps n!}{2(q-p)} \times \frac{1}{p!q!(n/2+C\sqrt{n})!(n/2-C\sqrt{n}-p-q)!} \geq \frac{\eps n!}{2(q-p)} \times \frac{\delta ^{1/2} n^{p+q}}{(n/2)!(n/2!)}\geq \delta \binom{n}{n/2} n^{p+q}$$
good pairs $A,B \in \mathcal A$. (In the last inequality we apply Fact~\ref{fact2}.) Since each such pair satisfies (ii), this completes the proof.
\endproof
Lemma~\ref{supertilt} can now be applied to prove the following \emph{container} lemma  which immediately implies Theorem~\ref{tilt}.

\begin{lemma}\label{tiltcont} Let $p,q\in \mathbb N$ be coprime with $p<q$.
There is a collection $\mathcal F \subseteq \P(n)$ with the following properties:
\begin{itemize}
\item[(i)]  $|\mathcal F|=2^{o(1)\binom{n}{n/2}}$;
\item[(ii)] If $\mathcal A \subseteq \P(n)$ is a $(p,q)$-tilted Sperner family, then $\mathcal A$ is contained in some member of $\mathcal F$;
\item[(iii)] $|F|\leq (q-p+o(1))\binom{n}{n/2}$ for every $F \in \mathcal F$.
\end{itemize}
\end{lemma}
\proof
Let $\eps >0$ and let $\delta$, $n_0$ be as in Lemma~\ref{supertilt}. Let $n \geq n_0$.
Define $G$ to be the graph with vertex set $\P(n)$ in which distinct sets $A$ and $B$ are adjacent if and only if $p|A \setminus B|=q|B\setminus A|$ or $q|A \setminus B|=p|B\setminus A|$. Thus a $(p,q)$-tilted Sperner family in $\P(n)$ is precisely an independent set in $G$. 

\begin{claim} 
There exists a function $f : \binom{V(G)}{\leq 2^n/\delta n} \to \binom{V(G)}{\leq (q-p+\eps)\binom{n}{n/2}}$  such that, for any independent set $I$ in $G$, there is a subset $S \subseteq I$ where 
$S \in \binom{V(G)}{\leq 2^n/\delta n} $
 and   $I \subseteq S \cup f(S)$. 
\end{claim}
To prove the claim,
fix an arbitrary total order $v_1, \dots, v_{2^n}$ on the vertices of $V(G)$. Given any independent set $I$ in $G$, define $G_0 := G$, and take $S$  to be initially empty. We add vertices to $S$  through the following iterative process:
 At Step $i$, let $u$ be the maximum degree vertex of $G_{i-1}$ (with ties broken by our fixed total order). If $u \notin I$ then define $G_i := G_{i-1} \setminus \{u\}$, and proceed to Step $i+1$. Alternatively, if $u \in I$ and $\deg_{G_{i-1}}(u) \geq \delta n$ then add $u$ to $S$, define $G_i := G_{i-1} \setminus (\{u\} \cup N_G(u))$, and proceed to Step $i+1$. Finally, if $u \in I$ and $\deg_{G_{i-1}}(u) < \delta n$, then set $f(S) := V(G_i)$ and terminate.

Observe that for any independent set $I $ in $G$ the process defined ensures that $S \subseteq I$ where $|S| \leq 2^{n}/\delta n$ and $I \subseteq S \cup f(S)$. Further, at the end of the process we know that
$\Delta (G_i) <\delta n$ and so $e(G_i)<  \delta n 2^n < \delta \binom{n}{n/2} n^{p+q}$. Hence, Lemma~\ref{supertilt} implies that $|f(S)|=|V(G_i)|\leq (q-p+\eps)\binom{n}{n/2}$.

To complete the claim we must show that $f$ is well-defined. That is, we must check that if the process described above yields the same set $S$ when applied to independent sets $I$ and $I'$, then it should also yield the same set $f(S)$. However, this is a consequence of the fact that we always chose $u$ to be the vertex of $I$ of maximum degree in $G_{i-1}$.  Thus, the claim is proven.

\medskip

Define $\mathcal F$ to be the collection of all the sets $S \cup f(S)$ for every $S \in \binom{V(G)}{\leq 2^n/\delta n} $. Then  (i) and (ii) hold and $|F|\leq (q-p+\eps)\binom{n}{n/2}+2^n/\delta n
\leq (q-p+2\eps)\binom{n}{n/2}$ for every $F \in \mathcal F$, as desired.

\endproof

\section{The number of $t$ error correcting codes and $2$-$(n,k,d)$-codes}\label{sec:code}
\subsection{Counting $t$ error correcting codes: Sapozhenko's question}
\label{errorcodescountingsubsection}
The \emph{Hamming distance} $d(A,B)$ between two sets $A,B\subseteq [n]$ is defined as 
$$d(A,B):=|A\backslash B|+|B\backslash A|.$$ 
In this section we will view a subset $A$ of $[n]$ as a string of length $n$ over the alphabet $\{0,1\}$ where the $i$th entry of the string is $1$ precisely when $i \in A$. 
In this setting, the Hamming distance between $A$ and $B$ can be rewritten as $$d(A,B)=|\{1\leq i \leq n : A_i\ne B_i\}|,$$ where for example, $A_i$ denotes the $i$th term in the string $A$.
The \emph{Hamming ball} of radius $r$ around a set $A\subseteq [n]$ is defined as 
$$B(A,r):=\{X\subseteq [n] : d(A,X)\leq r\}.$$
Given a family $\mathcal{C}\subset \mathcal{P}(n)$, we say $\mathcal{C}$ is a \emph{distance $d$ code} if the Hamming distance between any two distinct members of $\mathcal{C}$ is at least $d$. Moreover $\mathcal{C}$ is said to be a \emph{$t$ error correcting code} if there exists a \emph{decoding function} $\text{Dec}:\{0,1\}^n\to\mathcal{C}$ such that for every $X\in\{0,1\}^n$ and $A\in\mathcal{C}$ with $d(X,A)\leq t$ we have $\text{Dec}(X)=A$. Recall that $\mathcal{C}$ is $t$ error correcting if and only if for every pair $A,B\in\mathcal{C}$ we have $d(A,B)\geq 2t+1$. So $t$ error correcting codes are precisely distance $2t+1$ codes, i.e. codes where the Hamming balls $\{B(A,t):A\in\mathcal{C}\}$ are disjoint.  Since most communication channels are subject to channel noise, which can cause errors in the transmission of messages, the additional redundancy given by error correcting codes plays a crucial role in ensuring that the receiver can recover the original message. Given the widespread usage of such codes in digital communications, natural question to ask is  \emph{how many}  $t$ error correcting codes there are in total, i.e. estimate the size of $$|\{\mathcal{C}\subseteq \mathcal{P}(n) : \mathcal{C} \text{ is } t \text{ error correcting}\}|.$$

This problem was first raised by Sapozhenko \cite{saponew}.
We wish to bound the number of $t$ error correcting codes of length $n$ and alphabet $\{0,1\}$. An upper bound for the size of such a code is given by the Hamming bound, which gives an important limitation on the efficiency of error correcting codes. Let $V(n,t)$ be the volume of a Hamming ball of radius $t$ in $[n]$, so $V(n,t)=\sum_{k=0}^t \binom{n}{k}$. Then the Hamming bound states that if $\mathcal{C}$ is $t$ error correcting,  since the Hamming balls of radius $t$ centered at the members of $\mathcal{C}$ have to be disjoint, we have $$|\mathcal{C}|\leq \frac{2^n}{V(n,t)}.$$

If $\mathcal{C}$ attains equality in the Hamming bound, we say $\mathcal{C}$ is a \emph{perfect code}. Perfect codes are precisely those for which the Hamming balls centered at the codewords fill up the entire space $\{0,1\}^n$ without overlap. The trivial perfect codes are codes consisting of a single codeword (when $t=n$), or the whole of $\{0,1\}^n$ (when $t=0$), and \emph{repetition codes} where the same substring is repeated an odd number of times. The non-trivial perfect codes over prime-power alphabets must have the same parameters as the so-called \emph{Hamming codes} or the \emph{Golay codes} (see \cite{nontrivialperfectcodes}).

Let $H(n,t):=\frac{2^n}{V(n,t)}$. Since every subset of a $t$ error correcting code is also $t$ error correcting,  if $\mathcal{C}$ is $t$ error correcting then the number of $t$ error correcting codes is at least $2^{|\mathcal{C}|}$. In particular, if the parameters $n,t$ are such that a perfect code exists, then the number of $t$ error correcting codes is at least $2^{H(n,t)}$. Our first goal is to prove a corresponding upper bound:
\begin{theorem}\label{numberoferrorcodes}
Let $t=t(n)\ll\sqrt[3]{\frac{ n}{\log^2 n}}$. Then the number of $t$ error correcting codes is at most $2^{H(n,t)(1+o(1))}$. 
\end{theorem}

The range of $t$ given in Theorem \ref{numberoferrorcodes} is probably not optimal - indeed our guess is that the conclusion of Theorem \ref{numberoferrorcodes} should hold whenever $t\ll\frac{n}{\log n}$. However, a  heuristic argument suggests that if $t\gg\frac{n}{\log n}$ the conclusion of Theorem \ref{numberoferrorcodes} may fail. Indeed, suppose one could partition $\{0,1\}^n$ into disjoint copies of balls of radius $t+1$, obtaining roughly $\frac{t}{n}H(n,t)$ balls. From each ball we can pick one element, that is either the centre of the ball or an element at distance one from the centre, giving $n+1$ choices for each ball. Every family we obtain like this is a $t$ error correcting code, and we have roughly $n^{\frac{t}{n}H(n,t)}\gg 2^{H(n,t)}$ such families.

Our overarching proof strategy is similar to the one used in the previous section. However, now we will employ a two phase strategy to construct our containers and as such we require two different supersaturation results. 
The first states that if $|\mathcal{C}|$ is slightly bigger than $H(n,t)$ then it contains many \emph{bad pairs}, i.e. pairs at distance less than $2t+1$. 
Let $W(t,d)$ be the size of the intersection of two Hamming balls of radius $t$ in $n$, the centers being distance $d$ apart. So $W(t,1)\geq W(t,d)$ for all $d\geq 2$, and $W(t,1) = 2V(n-1,t-1)$. 
 The key observation is that the volume of the intersection of two balls is significantly smaller than the volume of a single ball.

\begin{lemma}\label{supersatforhammingcodes}
Let $\mathcal{C}\subset \mathcal{P}(n)$. If $|\mathcal{C}|\geq H(n,t)+x$ then there are at least $x\frac{n}{2t}$ pairs $A,B\in\mathcal{C}$ that have Hamming distance at most $2t$.
\end{lemma}
\begin{proof}
For $X\in \{0,1\}^n$, let $K_X:=\{A\in\mathcal{C}:d(A,X)\leq t\}$. For $k\in\mathbb N$ set $S_k:=\{X\in \{0,1\}^n : |K_X|=k \}$. Then $\sum_k k|S_k|=|\mathcal{C}|V(n,t)\geq 2^n + x V(n,t)$. So the number of pairs in $\mathcal{C}$ of distance at most $2t$ is at least $$ \frac{1}{W(t,1)}\sum_k |S_k|\binom{k}{2}\geq \frac{1}{W(t,1)}\sum_k |S_k|(k-1) \geq x\cdot\frac{V(n,t)}{W(t,1)}= x\cdot \frac{V(n,t)}{2V(n-1,t-1)}\geq x \frac{n}{2t}.$$
\end{proof}
Our next supersaturation lemma considers sets of size at least $2H(n,t)$. Consider the graph $G$ with $V(G)=\mathcal{P}(n)$, where two distinct vertices $A,B$ are connected by an edge of colour $d(A,B)$ if they form a bad pair, i.e. their Hamming distance is at most $2t$. Define $$\alpha :=\frac{n}{10tH(n,t)}.$$

\begin{lemma}\label{generalcodeserrorsupersat}
Let $\mathcal{C}\subset \mathcal{P}(n)$. If $|\mathcal{C}|\geq 2H(n,t)$, then there is an $A\in\mathcal{C}$ such that its degree in $G[\mathcal{C}]$ is at least $\alpha |\mathcal{C}|$.
\end{lemma}
\begin{proof}
Let $E_i$ denote the number of pairs of vertices connected by an edge of colour $i$ in $G[\mathcal{C}]$ for all $i=1,\dots ,2t$, and let $E:=\sum _i E_i$.
Define $K_X$  as in the proof of Lemma~\ref{supersatforhammingcodes}. Note that
\begin{align}\label{eq1}
\sum_{k=1}^{2t}W(t,k)E_k = \sum_{X\in\{0,1\}^n}\binom{|K_X|}{2}
\end{align}
since both terms count the number of pairs $(X,(A,B))$ where $X \in \{0,1\}^n$, $A,B \in \mathcal C$ and $d(X,A),$ $ d(X,B) \leq t$.
The average value of $K_X$ over all $X \in \{0,1\}^n$ is $|\mathcal C|V(n,t)/2^n$. Thus,
\begin{align}\label{eq2}
 \sum_{X\in\{0,1\}^n}\binom{|K_X|}{2} \geq 2^n \binom{|\mathcal{C}|V(n,t)/2^n}{2}.
\end{align}
Combining (\ref{eq1}) and (\ref{eq2}), 
 since $|\mathcal{C}|V(n,t)/2^n\geq 2$, we have that  $$\sum_{k=1}^{2t}W(t,k)E_k\geq \frac{|\mathcal{C}|^2 V(n,t)^2}{10\cdot 2^n}.$$
As $W(t,k)\leq W(t,1)= 2V(n-1,t-1)\leq \frac{2t}{n}V(n,t)$, we have that $$E\geq\frac{|\mathcal{C}|^2 V(n,t)n}{20t  2^n}$$ and the result follows.
\end{proof}

Given these two supersaturation results, we are now ready to prove the following container lemma which immediately implies Theorem \ref{numberoferrorcodes}.

\begin{lemma}\label{hammingcont} Let $t=t(n)\ll \sqrt[3]{\frac{ n}{\log^2 n}}$. There is a collection $\mathcal F \subseteq \P(n)$ with the following properties:
\begin{itemize}
\item[(i)]  $|\mathcal F|=2^{o(H(n,t))}$;
\item[(ii)] If $\mathcal C \subseteq \P(n)$ is a $t$ error correcting code, then $\mathcal C$ is contained in some member of $\mathcal F$;
\item[(iii)] $|F|\leq (1+o(1))H(n,t)$ for every $F \in \mathcal F$.
\end{itemize}
\end{lemma}
\proof
Let $0< \eps <1$ and let $n$ be sufficiently large. Let $G$ be the graph with vertex set $\P(n)$ in which distinct sets $A$ and $B$ are adjacent if and only if their Hamming distance is at most $2t$. Thus a $t$ error correcting code in $\P(n)$ is precisely an independent set in $G$. 

\begin{claim}\label{claimc} 
There exists a function $f : \binom{V(G)}{\leq \eps\frac{H(n,t)}{t\log n}} \to \binom{V(G)}{\leq (1+\eps)H(n,t)}$  such that, for any independent set $I$ in $G$, there is a subset $S \subseteq I$ where 
$S \in \binom{V(G)}{\leq \eps\frac{H(n,t)}{t\log n}} $
 and   $I \subseteq S \cup f(S)$. 
\end{claim}
To prove the claim, fix an arbitrary total order $v_1, \dots, v_{2^n}$ on the vertices of $V(G)$. Given any independent set $I$ in $G$, define $G_0 := G$, and take $S$  to be initially empty. We add vertices to $S$  through the following iterative process:
 At Step $i$, let $u$ be the maximum degree vertex of $G_{i-1}$ (with ties broken by our fixed total order). If $u \notin I$ then define $G_i := G_{i-1} \setminus \{u\}$, and proceed to Step $i+1$. Alternatively, if $u \in I$ and $\deg_{G_{i-1}}(u) \geq \eps {n}/{4t}$ then add $u$ to $S$, define $G_i := G_{i-1} \setminus (\{u\} \cup N_G(u))$, and proceed to Step $i+1$. Finally, if $u \in I$ and $\deg_{G_{i-1}}(u) <  \eps {n}/{4t}$, then set $f(S) := V(G_i)$ and terminate.

Observe that for any independent set $I $ in $G$ the process defined ensures that $S \subseteq I$  and $I \subseteq S \cup f(S)$. Further, at the end of the process we know that
$\Delta (G_i) < \eps {n}/{4t}$ and so $e(G_i)< |V(G_i)| \eps {n}/{4t}$. Hence, Lemma~\ref{supersatforhammingcodes} implies that $|f(S)|=|V(G_i)|\leq (1+\eps)H(n,t)$. Moreover it is easy to see that $f$ is well-defined.

To complete the proof of the claim, it 
remains to prove that $|S|\leq \eps {H(n,t)}/({t\log n})$.
We will distinguish two stages in the above algorithm, according to the size of $V(G_i)$. Let $S_1$ denote the set of vertices $u \in S$ that were added to $S$ in some Step $i$ of the algorithm where $|V(G_{i-1})|\geq 2 H(n,t)$. Set $S_2 :=S \setminus S_1$.
So there is some $k$ such that, up to and including Step $k$, every vertex added to $S$ lies in $S_1$, and every vertex added to $S$ after Step~$k$ lies in $S_2$.

By Lemma~\ref{generalcodeserrorsupersat}, for every $i \leq k$, at Step~$i$ we remove at least an $\alpha$ proportion of the vertices from $G_{i-1}$ to obtain $G_i$.
Thus, $|S_1|=k$ and $(1-\alpha)^k 2^n\leq 2H(n,t)$. Note that $\alpha \rightarrow 0$ as $n \rightarrow \infty$, so as $n$ is sufficiently large we have that $\alpha \leq 10 \log (1/(1-\alpha))$. 
Therefore,
$$|S_1|\leq \frac{\log \left(\frac{2^n}{2H(n,t)}\right)}{\log\left(\frac{1}{1-\alpha}\right)}\leq 10\frac{\log V(n,t)}{\alpha}\leq 5000\frac{tH(n,t)}{n}t\log(n/t)\leq \frac{\eps}{2} \frac{H(n,t)}{t\log n}.$$
Note that in the last inequality we use that $t\ll \sqrt[3]{\frac{ n}{\log^2 n}}$.

After Step $k$ we remove at least 
${\eps n}/{4t}$ vertices at each step, so we have $$|S_{2}|\leq \frac{8tH(n,t)}{\eps n}\leq \frac{\eps}{2}\frac{H(n,t)}{t\log n}.$$ 
Hence,  $$|S|=|S_1|+|S_2|\leq \eps\frac{H(n,t)}{t\log n},$$ as required. This finishes the proof of the claim.

\medskip

Define $\mathcal F$ to be the collection of all the sets $S \cup f(S)$ for every $S \in \binom{V(G)}{\leq \eps\frac{H(n,t)}{t\log n}} $. Then (ii) clearly holds. Further,
$$
|\mathcal{F}|  \leq \binom{2^n}{\leq  \eps\frac{H(n,t)}{t\log n}}\leq2^{2\eps\frac{H(n,t)}{t\log n}\log (t V(n,t)\log (n)/\eps)}\leq 2^{2\eps\frac{H(n,t)}{t\log n}(2t\log n + \log t + \log \log n + \log\frac{1}{\eps})}
\leq 2^{5 \eps H(n,t)}
$$
and $|F| \leq (1+2 \eps )H (n,t)$ for all $F \in \mathcal F$. Since $0<\eps <1$ was arbitrary, this proves the lemma.

\endproof

\subsection{Counting $2$-$(n,k,d)$-codes}
In this subsection, all pairs of sets considered are \emph{unordered}.
Let us now turn our attention to the space $\mathcal{Y}$ of  pairs of  disjoint $k$-subsets of $[n]$ (for some fixed $0<k \leq n/2$). 
 Given two pairs $(A_1,A_2)$, $(B_1,B_2) \in \mathcal Y$, the \emph{transportation distance}, or \emph{Enomoto--Katona distance}  is defined by $$d((A_1,A_2),(B_1,B_2)):=\min \{|A_1\setminus B_1|+|A_2 \setminus B_2|,|A_1\setminus B_2|+|A_2\setminus B_1|\}.$$
For convenience, throughout this subsection we will write \emph{distance} when we mean transportation distance.
The notion of transportation distance has been widely studied (also in a more general setting for metric spaces). See for example~\cite{vill} and the introduction of~\cite{bollobastransport} for background on the topic.

We say that a collection $\mathcal C \subseteq \mathcal Y$ is a \emph{$2$-$(n,k,d)$-code} if the distance between any two elements of $\mathcal C$ is at least $d$.
Write $C(n,k,d)$ for the maximum size of a $2$-$(n,k,d)$-code. Brightwell and Katona~\cite{bright} proved that
\begin{align}\label{Cbound}
C(n,k,d)\leq \frac{1}{2}\frac{n(n-1)\cdot \ldots\cdot (n-2k+d)}{\left ( k(k-1)\cdot\ldots\cdot \lceil \frac{d+1}{2}\rceil \right ) \left ( k(k-1)\cdot\ldots\cdot \lfloor \frac{d+1}{2}\rfloor\right )}=:H(n,k,d).
\end{align}
Recently the value of $C(n,k,d)$ has been determined for many values of $(n,k,d)$ (see~\cite{bollobastransport, chee}). As an example, we have equality or are `close' to equality in (\ref{Cbound}) when $k \geq 2$, $d =2k-1$ and  for certain (congruency) classes of $n$ (see~\cite{chee}). Further, the  bound in (\ref{Cbound}) is asymptotically sharp for fixed $k$, $d$ and $n \rightarrow \infty$ (see~\cite{bollobastransport}).
Our goal in this subsection is to prove the following  upper bound on the number of $2$-$(n,k,d)$-codes.

\begin{theorem}\label{numberoftransportcodes}
Suppose that $k=k(n)\leq n/2$ and  $t=t(n)\ll\sqrt[3]{\frac{k}{\log^2 n}}$ then the number of $2$-$(n,k,2t+1)$-codes is  at most $2^{H(n,k,2t+1)(1+o(1))}$.
\end{theorem}
Similarly to Theorem~\ref{numberoferrorcodes}, we believe that the correct range of $t$ in Theorem~\ref{numberoftransportcodes} should be $t\ll\frac{k}{\log n}$.

Given a pair $(A,B) \in \mathcal Y$, let $P((A,B),u)$ denote the family of pairs $(U,V )$ where $|U| =|V | = u$ and $U \subseteq A, V \subseteq B$ or vice versa. Then $|P((A,B),u)|=\binom{k}{u}^2$. Let $\mathcal{Z}(u)$ be the space of pairs of disjoint sets of size $u$ in $[n]$. So $|\mathcal{Z}(u)|=\frac12\binom{n}{u}\binom{n-u}{u}$ and note that $\mathcal Y=\mathcal Z(k)$.  We will refer to  $P((A,B),k-t)$ as the \emph{ball of radius $t$ around $(A,B)$}. In particular,  for any $(A_1,B_1),(A_2,B_2)$ in $\mathcal{Y}$, if $P((A_1,B_1),k-t)$ and $P((A_2,B_2),k-t)$ intersect, then $d((A_1,A_2),(B_1,B_2))\leq 2t.$

The proof strategy for Theorem~\ref{numberoftransportcodes} is extremely close to that of Theorem~\ref{numberoferrorcodes}. The following supersaturation lemma is an analogue of Lemma~\ref{supersatforhammingcodes}.
\begin{lemma}\label{t1} Let $\mathcal C \subseteq \mathcal{Y}$. If 
 $|\mathcal C|\geq H(n,k,2t+1)+x$ then there are at least $xk/t$ pairs $ (A_1,B_1),  (A_2,B_2)\in \mathcal C$ at distance at most $2t$.
\end{lemma}
\begin{proof}
Note that 
\begin{align}\label{b1}
\sum _{(A,B) \in \mathcal C}| P((A,B),k-t)| & \ge \binom{k}{k-t}^2\left(H(n,k,2t+1)+x\right)=\frac12\binom{n}{k-t}\binom{n-k+t}{k-t}+x\binom{k}{k-t}^2 \nonumber\\ & = |\mathcal{Z}(k-t)|+x\binom{k}{k-t}^2. 
\end{align}
Let $W(k-t,d)$ denote the largest possible intersection of two balls $P( (A_1,B_1),k-t)$ and $P( (A_2,B_2),k-t)$, amongst all $ (A_1,B_1),  (A_2,B_2)\in \mathcal Y$ with $d((A_1,B_1),(A_2,B_2))=d$. This is maximised when  
$A_1=A_2$ and $|B_1\cap B_2 |=k-1$, so $W(k-t,d)\leq W(k-t,1)$ for all $d\geq 2$. Now $W(k-t,1)=\binom{k}{k-t}\binom{k-1}{k-t}=\binom{k}{k-t}^2 \binom{k-1}{t-1}/\binom{k}{t}$.  Combining this with (\ref{b1}) we see that there are  at least
 $$ x{\binom{k}{t}}/{\binom{k-1}{t-1}}=x{k}/{t}$$ 
pairs $ (A_1,B_1),  (A_2,B_2)\in \mathcal C$ such that $P((A_1,B_1),k-t)$ and $P((A_2,B_2),k-t)$ intersect. Note that each such pair $ (A_1,B_1),  (A_2,B_2)\in \mathcal C$ have distance at most $2t$, as desired.
\end{proof}

    


Consider the graph $G$ with $V(G)=\mathcal{Y}$, two vertices $(A_1,B_1),(A_2,B_2) $ being connected by an edge of colour $d((A_1,B_1),(A_2,B_2) )$ if they form a bad pair, i.e. their transportation distance is at most $2t$. Define the constant $\alpha$ by $$\alpha :=\frac{k}{10tH(n,k,2t+1)}.$$

\begin{lemma}\label{cont2}
Let $\mathcal{C}\subset \mathcal{Y}$. If $|\mathcal{C}|\geq 2H(n,k,2t+1)$, then there is a vertex $(A_1,B_1)\in\mathcal{C}$ such that its degree in $G[\mathcal{C}]$ is at least $\alpha |\mathcal{C}|$.
\end{lemma}
\begin{proof}
One can prove the lemma by arguing in a similar way to the proof of Lemma~\ref{generalcodeserrorsupersat}. Now though given $X \in \mathcal Z(k-t)$ we take $K_X:=\{(A,B) \in \mathcal C : X \in P((A,B), k-t)\}$
and $S_i:= \{X  \in \mathcal Z(k-t): |K_X|=i\}$. By arguing  as in Lemma~\ref{generalcodeserrorsupersat} and using that $|\mathcal{Z}(k-t)|=H(n,k,2t+1)\binom{k}{k-t}^2$
we have that the number $E$ of edges in $G[\mathcal{C}]$ satisfies
$$E\geq \frac{|\mathcal{Z}(k-t)|}{W(k-t,1)} \binom{|\mathcal{C}|/H(n,k,2t+1)}{2}\geq |\mathcal{C}|^2\frac{k}{20tH(n,k,2t+1)}$$ and the result follows.
(In the last inequality we use that  $W(k-t,1)=\binom{k}{k-t}^2 \binom{k-1}{t-1}/\binom{k}{t}$.)
\end{proof}

The following container lemma immediately implies Theorem~\ref{numberoftransportcodes}; its proof follows the same approach used in the proof of Lemma~\ref{hammingcont}.
\begin{lemma}\label{hammingcont2} Let $k=k(n)\leq n/2$ and $t=t(n)\ll \sqrt[3]{\frac{ k}{\log^2 n}}$. There is a collection $\mathcal F $ of subsets of $\mathcal Y$ with the following properties:
\begin{itemize}
\item[(i)]  $|\mathcal F|=2^{o(H(n,k,2t+1))}$;
\item[(ii)] If $\mathcal C \subseteq \mathcal Y$ is a $2$-$(n,k,2t+1)$-code, then $\mathcal C$ is contained in some member of $\mathcal F$;
\item[(iii)] $|F|\leq (1+o(1))H(n,k,2t+1)$ for every $F \in \mathcal F$.
\end{itemize}
\end{lemma}
\begin{proof}
Let $0< \eps <1$ and let $n$  be sufficiently large. Let $G$ be the graph  defined before Lemma~\ref{cont2}.

\begin{claim}\label{claimt}
There exists a function $f : \binom{V(G)}{\leq \frac{\eps H(n,k,2t+1)}{t\log n}} \to \binom{V(G)}{\leq (1+\eps)H(n,k,2t+1)}$  such that, for any independent set $I$ in $G$, there is a subset $S \subseteq I$ where 
$S \in \binom{V(G)}{\leq \frac{\eps H(n,k,2t+1)}{t\log n}} $
 and   $I \subseteq S \cup f(S)$. 
\end{claim}
To prove the claim we argue as in Claim~\ref{claimc} except that we now apply the graph container algorithm with parameter $\eps k/2t$ instead of $\eps n/4t$. That is, at Step~$i$ if $u\notin I$ then set $G_i=G_{i-1}\setminus \{u\}$;
 if $u \in I$ and $\deg_{G_{i-1}}(u) \geq \eps {k}/{2t}$ we add $u$ to $S$, define $G_i := G_{i-1} \setminus (\{u\} \cup N_G(u))$; if $u \in I$ and $\deg_{G_{i-1}}(u) <  \eps {k}/{2t}$, set $f(S) := V(G_i)$ and terminate.

As before we have that for any independent set $I $ in $G$ the process defined ensures that $S \subseteq I$  and $I \subseteq S \cup f(S)$. Further, at the end of the process we know that
$\Delta (G_i) < \eps {k}/{2t}$ and so $e(G_i)< |V(G_i)| \eps {k}/{2t}$. Hence, Lemma~\ref{t1} implies that $|f(S)|=|V(G_i)|\leq (1+\eps)H(n,k,2t+1)$. Moreover $f$ is well-defined.

To complete the proof of the claim, it 
remains to prove that $|S|\leq \eps {H(n,k,2t+1)}/({t\log n})$.
As in Claim~\ref{claimc} we distinguish two stages in the above algorithm, according to the size of $V(G_i)$. Let $S_1$ denote the set of vertices $u \in S$ that were added to $S$ in some Step $i$ of the algorithm where $|V(G_{i-1})|\geq 2 H(n,k,2t+1)$. Set $S_2 :=S \setminus S_1$.
So there is some $k$ such that, up to and including Step $k$, every vertex added to $S$ lies in $S_1$, and every vertex added to $S$ after Step~$k$ lies in $S_2$.

By Lemma~\ref{cont2}, for every $i \leq k$, at Step~$i$ we remove at least an $\alpha$ proportion of the vertices from $G_{i-1}$ to obtain $G_i$.
Thus, $|S_1|=k$ and $(1-\alpha)^k |\mathcal Y|\leq 2H(n,k,2t+1)$. Note that \COMMENT{AT: I was originally worried that we needed  $k=o(n)$ for this... but think I was being stupid... can you just double check carefully please Adam} $\alpha \rightarrow 0$ as $n \rightarrow \infty$, so as $n$ is sufficiently large we have that $\alpha \leq 10 \log (1/(1-\alpha))$. 
Therefore,
\begin{equation*}
\begin{split}
|S_1| & \leq\frac{\log\left(\frac{|\mathcal{Y}|}{2H(n,k,2t+1)}\right)}{\log\left(\frac{1}{1-\alpha}\right)}\leq 10\frac{\log\left(\frac{(n-2k+2t)\ldots(n-2k+1)}{2(t!)^2}\right)}{\alpha}\leq 5000\frac{tH(n,k,2t+1)\log\binom{n}{t}}{k}\\
&\leq 10000\frac{t^2\log n}{k}H(n,k,2t+1)<\frac{\eps}{2}\frac{H(n,k,2t+1)}{t\log n}.
\end{split}
\end{equation*}
In the first inequality we used that $|\mathcal Y|=\binom{n}{k}\binom{n-k}{k}/2$ and
in the last inequality we use that $t\ll \sqrt[3]{\frac{ k}{\log^2 n}}$.

After Step $k$ we remove at least 
${\eps k}/{2t}$ vertices at each step, so we have 
$$|S_2|\leq \frac{2H(n,k,2t+1)}{\frac{\eps k}{2t}}=\frac{4t}{\eps k}H(n,k,2t+1)\leq \frac{\eps}{2} \frac{H(n,k,2t+1)}{t\log n}.$$
Hence,  $$|S|=|S_1|+|S_2|\leq \eps\frac{H(n,k,2t+1)}{t\log n},$$ as required. This finishes the proof of the claim.

\medskip

Define $\mathcal F$ to be the collection of all the sets $S \cup f(S)$ for every $S \in \binom{V(G)}{\leq \eps\frac{H(n,k,2t+1)}{t\log n}} $. Then (ii) clearly holds. Further,
\begin{equation*}
\begin{split}
|\mathcal{F}| & \leq\binom{\frac{1}{2}\binom{n}{k}\binom{n-k}{k}}{\leq\eps\frac{H(n,k,2t+1)}{t\log n}}\leq2^{2\eps\frac{H(n,k,2t+1)}{t\log n}\log \left(\binom{n}{t}^2t\log (n)/\eps\right)}\leq2^{10\eps\frac{H(n,k,2t+1)}{t\log n}(t\log n + \log t + \log \log n + \log\frac{1}{\eps})}\\
&\leq 2^{20\eps H(n,k,2t+1)},
\end{split}
\end{equation*}
and $|F| \leq (1+2 \eps )H (n,k,2t+1)$ for all $F \in \mathcal F$. Since $0<\eps <1$ was arbitrary, this proves the lemma.
\end{proof}

\section{A random version of Katona's intersection theorem}\label{sec:int}
A family $\mathcal A \subseteq \P(n)$ is \emph{$t$-intersecting} if $|A\cap B|\geq t$ for all $A,B \in \mathcal A$. In the case when $t=1$ we simply say that $\mathcal A$ is \emph{intersecting}.
Two of the most fundamental results in extremal set theory concern $t$-intersecting sets. 
The cornerstone theorem of Erd\H{o}s--Ko--Rado states that for every $k,t$ there exists an $n_0=n_0(k,t)$ such that if $n\geq n_0$ then the largest $t$-intersecting  $k$-uniform family is the trivial family, i.e., there is a $t$-element set which is contained in each of the sets.
The other fundamental theorem is Katona's intersection theorem \cite{katona}, which determines the size $K(n,t)$ of the largest $t$-intersecting (not necessarily uniform) family in $\mathcal{P}(n)$: it states that
$$K(n,t)=\begin{cases}
\binom{n}{\geq (n+t)/2}& \text{if $2|(n+t)$;}\\
2\binom{n-1}{\geq (n+t-1)/2} & \text{otherwise.}
\end{cases}$$
In the case when $n+t$ is even, $\binom{[n]}{\geq (n+t)/2}$ is a $t$-intersecting set of size $K(n,t)$. When $n+t$ is odd, $\binom{[n]}{\geq (n+t+1)/2} \cup  \binom{[n-1]}{(n+t-1)/2}$ is a $t$-intersecting set of size $K(n,t)$.
Notice that if $t=o(\sqrt{n})$ then $K(n,t)\sim 2^{n-1}$.

Beginning with the work of Balogh, Bohman and Mubayi~\cite{bbm}, the problem of
developing a 	`random' version of the Erd\H{o}s--Ko--Rado theorem has received significant attention (see~\cite{bbm, das, gauy, kahn1, kahn2}). In this section, we raise the analogous question for Katona's intersection theorem. More precisely,
let $\P(n, p)$ be the set obtained from $\P(n)$ by selecting elements randomly with probability $p$ and independently of all other choices. 
\begin{question}\label{ques1}
Suppose that $n \in \mathbb N$, $t=t(n)\in \mathbb N$ and write $K:=K(n,t)$. For which values of $p$ do we have that, with high probability, the largest $t$-intersecting family in $\P(n,p)$ has size $(1+o(1))pK$?
\end{question}
The model $\P(n,p)$ was first investigated by R\'enyi~\cite{renyi} who determined the probability threshold for the property that $\P(n,p)$ is not itself an antichain, thereby answering a question of Erd\H{o}s. 
More recently, a random version of Sperner's theorem for $\P(n, p)$ was obtained independently by Balogh, Mycroft and Treglown~\cite{sper} and by Collares Neto and  Morris~\cite{cm}.

In this section we give a precise answer to Question~\ref{ques1} in the case when $t=o(\sqrt{n})$. 
For intersecting families (i.e. for the $t=1$ case), this question has also been resolved independently by Mubayi and  Wang~\cite{dhruv}.
Clearly the conclusion of Question~\ref{ques1} is not satisfied if $p<C/2^{n}$ for any constant $C>0$. The next result implies that the conclusion of Question~\ref{ques1} is not satisfied for $p=2^{-\Omega (\sqrt{n}\log n)}$ and $t=o(\sqrt{n})$.
\begin{theorem}\label{katlower}
Let $p=2^{-\Omega (\sqrt{n}\log n)}$ where $p\geq \omega (n)/2^{n}$ for some function $\omega (n)\rightarrow \infty$ as $n \rightarrow \infty$, and let $t=o(\sqrt{n})$. Then there exists a constant $\eps>0$ such that, with high probability, the largest $t$-intersecting family in $\mathcal P(n,p)$ has size at least $(\frac{1}{2}+\eps)2^np$.
\end{theorem}
\proof
The choice of $p$ and $t$ implies that there exists a constant $a>0$ such that $p<2^{-a\sqrt{n}\log n}$ and $t<\frac{a}{100}\sqrt{n}$ for $n$ sufficiently large. 
Define $\eps$ so that $0<\eps \ll a$.

Let $\mathcal A$ denote the set of elements $A$ of $\mathcal P(n)$ that satisfy $n/2-a\sqrt{n}/2 \leq |A| \leq n/2-a\sqrt{n}/4$ and $|A\cap[n/2]|\geq n/4+t/2$. The latter condition implies that $\mathcal A$ is a $t$-intersecting family.
\begin{claim}\label{aclaim}
$|\mathcal A| \geq 4\eps 2^n.$
\end{claim}
The claim holds since
\begin{equation*}
\begin{split}
|\mathcal{A}|&=\sum_{s=a\sqrt{n}/4}^{a\sqrt{n}/2}\sum_{k=0} ^{n/4-t/2-s} \binom{n/2}{n/4+t/2+k}\binom{n/2}{n/4-t/2-s-k}\\
& \geq \frac{a\sqrt{n}}{4} \sum_{k=0} ^{a\sqrt{n}} \binom{n/2}{n/4+t/2+k}\binom{n/2}{n/4-t/2-a\sqrt{n}/2-k} \\
& \geq \frac{a\sqrt{n}}{4}  \cdot a\sqrt{n} \binom{n/2}{n/4+2a\sqrt{n}}\binom{n/2}{n/4-2a\sqrt{n}} \geq 4\eps 2^n,
\end{split}
\end{equation*}
where the last inequality follows by applying Facts~\ref{fact1} and~\ref{fact2} and since $\eps \ll a$.

\smallskip

Write $\mathcal A_{\rm ex}:=\binom{[n]}{\geq n/2+t/2}$. So $\mathcal A_{\rm ex}$ is a $t$-intersecting set.
Since $t =o(\sqrt{n})$ note that $|\mathcal A_{\rm ex}|\geq (1/2-\eps/2)2^n$. As $p\geq \omega(n)/2^n$,
by the Chernoff bound for the binomial distribution, we have that, with high probability, 
$\P(n, p)$ contains at least $(1/2-\eps)p2^n$ elements from $\mathcal A_{\rm ex}$. Denote this set by $\mathcal A_{{\rm ex},p}$. 
We will show that, with high probability, we can add a significant number of elements from $\mathcal A$ to $\mathcal A_{{\rm ex},p}$ to obtain a $t$-intersecting set in $\P(n,p)$ of size at least $(\frac{1}{2}+\eps)2^np$.

Consider any $A \in \mathcal A$. The number of elements $B \in \binom{[n]}{\geq n/2}$ with $|A\cap B| < t$ is 
\begin{align}
\sum _{k=0} ^{t-1} \binom{|A|}{k} \binom{n-|A|}{\geq n/2-k} & \leq  \binom{|A|}{t-1} \sum _{k=0} ^{t-1}  \binom{n-|A|}{\geq n/2-k} 
 \leq  \binom{n/2-a\sqrt{n}/4}{t-1} \sum _{k=0} ^{t-1}  \binom{n/2+a\sqrt{n}/2}{\geq n/2-k} \nonumber  \\ & =  \binom{n/2-a\sqrt{n}/4}{t-1} \sum _{k=0} ^{t-1}  \binom{n/2+a\sqrt{n}/2}{\leq a\sqrt{n}/2+k} \nonumber
\\ & \leq 2 \binom{n/2-a\sqrt{n}/4}{t-1}  \binom{n/2+a\sqrt{n}/2}{ a\sqrt{n}/2+t}. \nonumber
\end{align} 
Further, 
$$ 2 \binom{n/2-a\sqrt{n}/4}{t-1}  \binom{n/2+a\sqrt{n}/2}{ a\sqrt{n}/2+t} \leq n^t \left(\frac{3\sqrt{n}}{a}\right)^{0.55a\sqrt{n}}\leq 2^{0.6a\sqrt{n}\log n},$$
where in the first inequality we use that $t < a\sqrt{n}/100$ and apply Fact~\ref{fact3}.

Let $\mathcal A_p$ denote the set of elements $A \in \mathcal A$ that lie in $\P(n,p)$ and where $\mathcal A_{{\rm ex},p}\cup \{A\}$ is a $t$-intersecting set.
Thus, the probability that $A\in \mathcal A$ lies in $\mathcal A_p$ is at least
$$p(1-p) ^{2^{0.6a\sqrt{n}\log n}}.$$
By $X$ denote the size of the family $\mathcal A_p$. 
By Claim~\ref{aclaim}, 
$$\mathbb E(X) \geq 4\eps  2^n p (1-p)^{2^{0.6a\sqrt{n}\log{n}}}\geq 4\eps p2^{n}(1-p2^{0.6a\sqrt{n}\log{n}})\geq 4\eps p2^{n}(1-2^{-0.4a\sqrt{n}\log{n}})\geq 3\eps p2^{n},$$ where the last inequality follows since $n$ is sufficiently large.

Write $\mathcal A=\{A_1,\dots, A_m\}$ and $X=\sum _{i=1} ^m X_i$ where $X_i=1$ if $A_i \in \mathcal A_p$ and $X_i=0$ otherwise. 
Note that the random variables $X_i,X_j$ are not independent if and only if there is some $B \in \mathcal A_{\rm ex}$ such that $|B \cap A_i|,|B\cap A_j|<t$. In this case, $|B|\geq n/2$ and so $|A_i \cup A_j|\leq n/2+2t \leq n/2 +a \sqrt{n}/50.$
Further, $|A_i|,|A_j| \geq n/2 -a\sqrt{n}/2$ and thus $|A_i \setminus A_j|, |A_j \setminus A_i|\leq a \sqrt{n}$. 
So given a fixed $i$, the number of $X_j$s that are not independent with $X_i$ is at most 
$$\binom{n}{\leq a \sqrt{n}}\binom{|A_i|}{\leq a \sqrt{n}}\leq 2 \binom{n}{ a \sqrt{n}}^2 \leq 2\left ( \frac{e \sqrt{n}}{a} \right )^ {2a \sqrt{n}} < 2^ {10 a \sqrt{n} \log n}.$$
Write $i \sim j$ to mean that $X_i$ and $X_j$ are not independent.
By abusing notation  let us also write $A_i$ to denote the event that $A_i \in \mathcal A_p$.
Consider
$$\Delta := \sum _{i \sim j} \mathbb P (A_i  \mathcal  \cap A_j) .$$
For $A_i,A_j \in \mathcal A_p$ we require that $A_i,A_j \in \mathcal P(n,p)$ and so   $\mathbb P (A_i  \mathcal  \cap A_j) \leq p^2.$
Therefore,
$$\Delta \leq \sum _{i=1} ^{m} 2^ {10 a \sqrt{n} \log n} p^2 \leq 2 ^{n} 2^ {10 a \sqrt{n} \log n} p^2.$$
In particular,   $\Delta =o(\mathbb E(X)^2)$. Thus, by applying Corollary 4.3.4 from~\cite{alon} (Chebyshev's inequality) we have that, with high probability, $X \geq 2\eps p2^{n}$.

Note that $\mathcal A_p \cup \mathcal A_{{\rm ex}, p}$ is a $t$-intersecting set in $\P(n,p)$ and, with high probability, it has size at least  $(1/2+\eps)p2^n$, as required.
\endproof
By arguing precisely as in the proof of Theorem~\ref{katlower} we in fact obtain the following result for $t=O(\sqrt{n})$.
\begin{theorem}\label{katlower2}
Given any constant $C >0$, there is a constant $\eps>0$ such that the following holds.
Let $p<2^{-100C\sqrt{n}\log n}$ where $p\geq \omega (n)/2^{n}$ for some function $\omega (n)\rightarrow \infty$ as $n \rightarrow \infty$, and let $t\leq C\sqrt{n}$. Write $K:=K(n,t)$. Then there exists a constant $\eps>0$ such that, with high probability, the largest $t$-intersecting family in $\mathcal P(n,p)$ has size at least $(1+\eps)pK$.
\end{theorem}


The following result together with Theorem~\ref{katlower} resolves Question~\ref{ques1} for $t=o(\sqrt{n})$.

\begin{thm}
If $p=2^{-o(\sqrt{n}\log n)}$ and $t=o(\sqrt{n})$ then with high probability the largest $t$-intersecting family in $\mathcal P(n,p)$ has size $(\frac{1}{2}+o(1))2^np$.
\end{thm}
\begin{proof}

Note that for this range of $t$,  with high probability, the size of the largest $t$-intersecting family in $\mathcal P(n,p)$ is at least $(\frac{1}{2}+o(1))2^np$. Hence to prove the theorem, it suffices to show  that the largest \emph{intersecting} family  in $\mathcal P(n,p)$ has size at most $(\frac{1}{2}+o(1))2^np$. That is, it suffices to prove the upper bound in the theorem for $t=1$, since for any $t\geq 2$ any $t$-intersecting family is also $1$-intersecting (i.e. intersecting).

 Fix any $\delta>0$ and define $0<\eps \ll \gamma \ll \delta$. We will show that with high probability the largest intersecting family in $\mathcal P(n,p)$ has size at most $(\frac{1}{2}+\delta)2^np$.

The first step in the proof is to create a collection of \emph{containers} that house all intersecting families. Define the graph $G$ on vertex set $\mathcal{P}(n)$ where distinct $A,B$ are adjacent in $G$ precisely if $A\cap B=\emptyset$. In order to bound the size of the containers, we require a supersaturation result. 

\begin{claim}\label{cd}
 If $\mathcal{F}\subseteq \mathcal{P}(n)$ where $|\mathcal{F}|\geq \sum_{k=\frac{n}{2}-C\sqrt{n}}^n \binom{n}{k}$ for some constant $C>0$, then 
 $e(G[\mathcal{F}])\geq 2^{n+\frac{C}{20}\sqrt{n}\log n}$ and so
$\Delta(G[\mathcal{F}])\geq 2^{\frac{C}{20}\sqrt{n}\log n}$.
\end{claim}
A result of Frankl~\cite{frank} and Ahlswede~\cite{ash} implies that,
 given $|\mathcal{F}|$, the number of edges in $G[\mathcal{F}]$ is minimised if $\mathcal{F}$ consists of the top layers of $\mathcal P(n)$, and possibly one partial layer. That is, there are no $A,B\subseteq [n]$ with $|A|<|B|$ and $A\in\mathcal{F}$ but $B\notin \mathcal{F}$. 

Hence we may assume that $\mathcal{F}$ consists of the top $\frac{n}{2}+C\sqrt{n}+1$ layers of $\mathcal{P}(n)$. We will estimate the degrees of vertices in the lowest $C\sqrt{n}/2$ layers of $\mathcal{F}$. The total number of vertices in these layers is at least $\delta_1 2^n$, where $\delta_1>0$ is a constant dependent only on $C$. The degree of each vertex $v$ in these layers is bounded below by $$\deg _{G[\mathcal F]}(v)\geq \binom{n/2+(C/2)\sqrt{n}}{n/2-(C/2)\sqrt{n}}\geq 2^{(C \sqrt{n}\log{n})/10}.$$
Thus, the number of edges in $G[\mathcal{F}]$ is at least $\delta_1 2^n 2^{(C\sqrt{n}\log n)/10}/2\geq 2^{n+\frac{C}{20}\sqrt{n}\log n}$, thereby proving the claim.


\medskip

By applying the graph container algorithm to $G$ with parameter $2^{\eps \sqrt{n} \log n}$, Claim~\ref{cd} implies that there is a function $f: \binom{V(G)}{\leq 2^{n-\eps \sqrt{n} \log n}} \rightarrow \binom{V(G)}{\leq (1/2+\gamma )2^n}$
such that, for any independent set $I$ in $G$, there is a subset $S \subseteq I$ where $S \in \binom{V(G)}{\leq 2^{n-\eps \sqrt{n} \log n}}$ and $I \subseteq S\cup f(S)$. Note  here we used that
$$ \sum_{k=\frac{n}{2}-20\epsilon\sqrt{n}}^n \binom{n}{k}\leq \left(\frac{1}{2}+\gamma\right)2^n.$$
Let $\mathcal F$ be the collection of all sets $S \cup f(S)$ for all $S \in \binom{V(G)}{\leq 2^{n-\eps \sqrt{n} \log n}}$. So  $|F|\leq (1/2+2\gamma )2^n$ for every $F \in \mathcal F$.
Further, Fact~\ref{fact3} implies that
\begin{align}
\log|\mathcal{F}|& \leq\log\left(\sum_{a\leq2^{n-\epsilon\sqrt{n}\log{n}}}\binom{2^n}{a} \right)\leq  \log \left(2\left(\frac{e2^n}{2^{n-\eps\sqrt{n}\log n}}\right)^{2^{n-\eps\sqrt{n}\log n}}\right) \nonumber\\ & 
\leq  n2^{n-\eps\sqrt{n}\log n}\leq 2^{n-\frac{\eps}{2}\sqrt{n}\log n}.\label{eq11}
\end{align}
Given any $F\in \mathcal F$, by the Chernoff bound for the binomial distribution we have that
\begin{align}\label{eq12}\mathbb P\bigl(|F \cap \mathcal P(n,p)|\geq (1/2+4 \gamma )2^n p\bigr) \leq 2 e^{-\gamma ^2 2^n p/2}.
\end{align}
Thus, (\ref{eq11}), (\ref{eq12}) and the choice of $p$ imply that with high probability $|F \cap \mathcal P(n,p)|\leq (1/2+\delta)2^n p$ for all $F \in \mathcal F$. Since every intersecting family in $\mathcal P(n)$ lies in some $F \in \mathcal F$, the 
theorem now follows.
\end{proof}




\section{Sperner's theorem revisited}
\subsection{Counting antichains in $\mathcal{P}(n)$}\label{seca}
  Sperner's theorem \cite{sperner}  states that the largest antichain  in $ \mathcal{P}(n)$ has size $\binom{n}{\lfloor n/2\rfloor}$. It was Dedekind \cite{dedekindsperner} in 1897 who first attempted to find the total number $A(n)$ of distinct antichains in $\mathcal{P}(n)$. 

Since every subset of an antichain is an antichain itself, it follows that  $2^{\binom{n}{\lfloor n/2\rfloor}}\leq A(n).$ 
The following result of Kleitman determines $A(n)$ up to an error term in the exponent.

\begin{thm}[Kleitman~\cite{kleitmansperner}]\label{numberofsperner}
The number of antichains in $\mathcal{P}(n)$ is $2^{\binom{n}{\lfloor n/2\rfloor}(1+o(1))}.$
\end{thm}

For further details on the history of this, and similar questions, we refer the reader to the brilliant survey by Saks~\cite{sakssurvey}. Our first goal in this section is to give an alternative proof of Theorem \ref{numberofsperner} using the container method. We will apply the following supersaturation result of Kleitman~\cite{kleitmanspernersupersat}.

\begin{thm}[Kleitman~\cite{kleitmanspernersupersat}]\label{spernersupersat}
Let $\mathcal{A}\subseteq\mathcal{P}(n)$ with $|\mathcal{A}|\geq \binom{n}{\lfloor n/2\rfloor}+x$. Then $\mathcal{A}$ contains at least $\left(\lfloor n/2\rfloor +1\right)x$ pairs $A,B$ with $A\subset B$.
\end{thm}
For $x\leq \binom{n}{\lfloor n/2\rfloor+1}$, Theorem~\ref{spernersupersat} is easily seen to be optimal, by taking a full middle layer and any $x$ sets on the layer above. Our proof of this supersaturation theorem will make use of the existence of a \emph{symmetric chain decomposition} (or SCD) of $\mathcal{P}(n)$, given first by de Bruijn, Tengbergen and Kruyswijk \cite{bruijnscd}. An SCD $\mathcal{X}$ is a partition of $\mathcal P (n)$ into \emph{symmetric chains}, i.e. chains that for some $k\leq n/2$ consist of precisely  one set of each size $i$ between $k$ and $n-k$. The proof we give is very similar to the proof of a more general result from~\cite{js}.

\begin{proof}[Proof of Theorem \ref{spernersupersat}]
Without loss of generality we may assume that $\emptyset, [n] \notin \mathcal A$.
Given any SCD $\mathcal{Z}$ we say  that $\mathcal{Z}$ \emph{contains} a bad pair $A,B$ if $A,B\in\mathcal{A}$ and there exists a chain $X\in\mathcal{Z}$ such that $A,B\in X$. Note that $\mathcal{Z}$ is a partition of $\mathcal{P}(n)$ into $\binom{n}{\lfloor n/2 \rfloor}$ chains, hence by the pigeonhole principle $\mathcal{Z}$ contains at least $x$ bad pairs. 

Fix some SCD $\mathcal X$.
Each permutation $\pi\in S_n$ induces a permutation on the subsets of $[n]$ and hence on collections of subsets of $[n]$. 
In particular, $\pi (\mathcal X)$ is a SCD.
We will pick a random permutation $\pi\in S_n$ and  estimate the number of bad pairs contained in $\pi(\mathcal{X})$. 

Let $\mathcal P$ denote the set of ordered pairs $A,B\in \mathcal A $ where $A\subset B$.
Consider any $(A,B) \in \mathcal P$.
If $|B|\geq \lfloor n/2\rfloor +1$ define   $\delta_A(B):=\{S\subset [n]:S\subset B, |S|=|A|\}$. Otherwise define $\delta_A(B):=\{S\subset [n]:A\subset S, |S|=|B|\}$.
Since  $A,B\notin \{\emptyset,[n]\}$, in both cases we have $|\delta_A(B)|\geq \binom{\lfloor n/2\rfloor +1}{\lfloor n/2\rfloor}$.

If $|B|\geq \lfloor n/2\rfloor +1$,
the probability that there is a chain $X\in\pi(\mathcal{X})$ with $S,B\in X$ is the same for all $S\in \delta_A(B)$. So the probability that there is a chain $X\in\pi(\mathcal{X})$ with $A,B\in X$ is at most $\frac{1}{\lfloor n/2 \rfloor +1}$. 
Similarly if $|B|\leq \lfloor n/2\rfloor $, the probability that there is a chain $X\in\pi(\mathcal{X})$ with $A,B\in X$ is at most $\frac{1}{\lfloor n/2 \rfloor +1}$. 
Thus, the expected number of bad pairs in $\pi(\mathcal{X})$ is at most $|\mathcal P|/(\lfloor n/2 \rfloor +1)$. On the other hand, as $\pi(\mathcal{X})$ is a SCD there are at least $x$ bad pairs in $\pi(\mathcal{X})$. Hence,
$|\mathcal P| \geq \left(\lfloor n/2\rfloor +1\right)x$, as desired.
\end{proof}

Now Theorem~\ref{numberofsperner} follows from an easy application of the container method. 
\begin{proof}[Proof of Theorem~\ref{numberofsperner}]
Let $\eps >0$ and let $G$ be the graph with vertex set $\mathcal P(n)$ where $A$ and $B$ are adjacent precisely if $A \subset B$ or $B \subset A$.
By applying the graph container algorithm to $G$ with parameter $\eps n/10$, Theorem~\ref{spernersupersat} implies that we obtain a function 
$f : \binom{V(G)}{\leq 10 \cdot 2^n/\eps n} \to \binom{V(G)}{\leq (1+\eps)\binom{n}{\lfloor n/2\rfloor }}$  such that, for any independent set $I$ in $G$, there is a subset $S \subseteq I$ where 
$S \in \binom{V(G)}{\leq 10 \cdot 2^n/\eps n} $
 and   $I \subseteq S \cup f(S)$. 

Let $\mathcal F$ be the collection of all sets $S \cup f(S)$ for all $S \in\binom{V(G)}{\leq 10 \cdot 2^n/\eps n}$. Then
$$|\mathcal{F}|\leq \binom{2^n}{\leq 10\frac{2^n}{\epsilon n}}\leq 2^{20\frac{2^n}{\epsilon n} \log n}=2^{o\left(\binom{n}{\lfloor n/2\rfloor}\right)}.$$
Further, every antichain is an independent set in $G$ and therefore lies in some element of $ \mathcal F$ and $|F|\leq (1+2\eps)\binom{n}{\lfloor n/2\rfloor }$ for every $F \in \mathcal F$. The existence of $\mathcal F$ immediately proves the theorem.
\end{proof}

Let $\mathcal{F}\subseteq \P (n)$, and for $i\in [n]$, let $B_i$ denote the number of comparable pairs $A,B\in\mathcal{F}$ with $|B\backslash A|=i$ and let $B_{\geq i}:=B_i+B_{i+1}+\ldots+B_n$. Then by arguing as in the proof of 
 Theorem~$\ref{spernersupersat}$ we get the following proposition. We include it because we strongly suspect it will have applications.

\begin{prop}\label{spernerhypersupersat}
Let $n,N,x\in \mathbb N$. Suppose $\mathcal{F}\subseteq \mathcal{P}(n)$ where $|\mathcal{F}|=\binom{n}{n/2}+x$ and for all $A\in\mathcal{F}$ we have $N\leq|A|\leq n-N$. Then $$\frac{B_{\geq N}}{\binom{\lfloor n/2\rfloor + \lceil N/2 \rceil}{N}}+\sum_{k=1}^{N-1}\frac{B_k}{\binom{\lfloor n/2\rfloor + \lceil k/2 \rceil}{k}} \geq x.$$
\end{prop}

\subsection{A two-coloured generalisation of Theorem~\ref{numberofsperner}}\label{secb}

Now we turn our attention to a two-coloured generalisation of Sperner's theorem, which was discovered independently by Katona \cite{katonaspernergener} and Kleitman \cite{kleitmanspernergener}. Given a (two-)colouring of $[n]$, we say that a pair of sets $A,B\in\mathcal{P}(n)$ is \emph{comparable with monochromatic difference} if $A\subset B$, and the difference $B\backslash A$ is monochromatic.

\begin{thm}[Katona \cite{katonaspernergener}, Kleitman \cite{kleitmanspernergener}]\label{generalisedspernerthm}
Let $\mathcal{A}\subseteq \mathcal{P}(n)$, and let $R\cup W$ be a partition of $[n]$ (i.e. a two-colouring of $[n]$ using colours Red and White). If $\mathcal{A}$ does not contain a pair of sets which are comparable with monochromatic difference then $|\mathcal{A}|\leq \binom{n}{\lfloor n/2\rfloor}$. 
\end{thm}

Note that setting $R=\emptyset$ in Theorem~\ref{generalisedspernerthm} gives the classical Sperner theorem. We will consider  the following question: given the partition $R\cup W$, how many families are there without two comparable sets whose difference is monochromatic? Alternatively, how many families are there for which there exists a partition $R\cup W$ such that there are no comparable pairs with monochromatic difference? (The answers to these two questions are at most a factor of $2^n$ apart.) To answer this question using the container method, first we need a supersaturation result.

\begin{lemma}\label{twocolourspernersupersat}
Let $\eps>0$ and $n$ be sufficiently large. Given a partition $R\cup W=[n]$ and a family $\mathcal{F}\subset \mathcal{P}(n)$ of size $|\mathcal{F}|\geq (1+\eps)\binom{n}{n/2}$, there are at least $\eps \binom{n}{n/2}\frac{n^{3/4}}{4}$  comparable pairs
$A,B \in \mathcal F$ with monochromatic difference.
\end{lemma}
We note that the factor ${n^{3/4}}/{4}$ is far from optimal  and indeed the proof below can easily be strengthened to some function of the form $n^{1-o(1)}$ instead of the $n^{3/4}$. The number of such pairs is probably at least $\eps\binom{n}{n/2}\left(\lfloor\frac{n}{2}\rfloor +1\right)$ (if $n>5$). We could not prove this, but fortunately this weaker result suffices to prove the counting theorem later (in fact one could replace the $n^{3/4}$ by anything bigger than $c\sqrt{n}\log n$ and the calculations would still go through, with a worse $o(1)$ error term in the final result).
\begin{proof}[Proof of Lemma \ref{twocolourspernersupersat}]
Suppose first that $|R|\leq n^{3/4}$. For each $S\subseteq R$, let $\mathcal{F}_S:=\{A\subseteq W: A\cup S\in\mathcal{F}\}$. 
Given any pair of sets $A,B$ with $A \subset B$ and $A,B \in \mathcal F_S$ we can find a comparable pair $A\cup S, B\cup S $ in $\mathcal F$ with monochromatic difference. Thus, Theorem~\ref{spernersupersat} implies that
the number  of comparable pairs in  $\mathcal F$ with monochromatic difference is at least 
\begin{equation*}
\begin{split}
 &  \sum_{S\subseteq R}\left(|\mathcal{F}_S|-\binom{n-|R|}{(n-|R|)/2}\right)\left(\frac{n-|R|}{2}+1\right)\geq \left((1+\eps) \binom{n}{n/2} - 2^{|R|}\binom{n-|R|}{(n-|R|)/2}\right)\frac{n}{3}\\
& \geq \left((1+\eps) \binom{n}{n/2} - \left(1+\frac{\eps}{2}\right) \binom{n}{n/2}\right)\frac{n}{3}=\eps\binom{n}{n/2}\frac{n}{6},
\end{split}
\end{equation*}
as required. Note that in the penultimate inequality we used Fact~\ref{fact1} and that $|R|=o(n)$.

We may therefore assume that $|R|\geq n^{3/4}$ and $|W|\geq n^{3/4}$. Remove all monochromatic elements of $\mathcal{F}$, and all those elements of $\mathcal{F}$ that contain the entire set $R$ or the entire set $W$ as subsets. The number of such sets is at most $\frac{\eps}{2}\binom{n}{n/2}$. Following the original proof of Theorem \ref{generalisedspernerthm}, let $\mathcal{B}_W$ and $\mathcal{B}_R$ be SCDs of $\mathcal{P}(W)$ and $\mathcal{P}(R)$. Let the group $G=S_{|W|}\times S_{|R|}$ act on $R\cup W$ in the natural way, permuting the elements in the two sets. From now on for simplicity we shall refer to comparable pairs in $\mathcal{F}$ with monochromatic difference simply as \emph{bad pairs}. We say that the pair of SCDs $\mathcal{B}_R$, $\mathcal{B}_W$  \emph{contains} the bad pair $(A,B)$ if there exist chains $X\in\mathcal{B}_R$ and $Y\in \mathcal{B}_W$ such that $Y$ contains $(A\cap W, B\cap W)$ and $X$ contains $(A\cap R,B\cap R)$. 

Let $x:=\frac{\eps}{2}\binom{n}{n/2}$. We first show that every pair of SCDs contains at least $x$ bad pairs. This follows instantly from the original proof of Theorem \ref{generalisedspernerthm}: suppose on the contrary, we could find a pair of SCDs $\mathcal{B}_R$ and $\mathcal{B}_W$ and  a family $\mathcal{A}\subset \mathcal{F}$ of size $|\mathcal{A}|=\binom{n}{n/2}+1$ such that the pair $\mathcal{B}_R,\mathcal{B}_W$ does not contain any bad pairs from $\mathcal{A}$. If a pair of chains $(X,Y)\in\mathcal{B}_R\times \mathcal{B}_W$ does not contain any bad pairs, then the number of sets $A$ such that $Y$ contains $A\cap W$ and $X$ contains $A\cap R$  is at most $\min \{|X|,|Y|\}$. So if $X_1\subset \ldots \subset X_t$ is a chain in $\mathcal{B}_R$ and $Y_1\subset \ldots \subset Y_s$ is a chain in $\mathcal{B}_W$ then $\mathcal{A}$ contains at most $\min\{s,t\}$ sets of the form $X_i\cup Y_j$, which is also the number of sets of this form having size exactly $\lfloor n/2\rfloor$. Hence  $\sum_{X\in\mathcal{B}_R, Y\in\mathcal{B}_W}\min \{|X|,|Y|\}=\binom{n}{n/2}$ because  both sides count the number of subsets of $[n]$ of size $\lfloor n/2\rfloor$. Thus for every subfamily $\mathcal{A}\subseteq \mathcal{F}$ with $|\mathcal{A}|=\binom{n}{n/2}+1$  there exists a pair of chains $X \in \mathcal B_R$, $Y \in \mathcal B_W$ containing a bad pair from $\mathcal{A}$, and the claim follows. 

Choose a random element $\pi \in G$.
We claim  the probability that $\pi(\mathcal{B}_R,\mathcal{B}_W)$ contains a given bad pair is at most ${2}/{n^{3/4}}$. To see this, let $(A,B)$ be a bad pair. Without loss of generality we may assume $A\subset B$ and $B\backslash A\subseteq R$. This implies that $B\cap W = A\cap W$, and since $A\neq B$ we have $A\cap R \neq B \cap R$. The probability that $(A,B)$ is contained in the pair $\pi(\mathcal{B}_R,\mathcal{B}_W)$ of SCDs is equal to the probability that $(A\cap R, B\cap R)$ is contained in $\pi_R(\mathcal{B}_R)$ (where $\pi_R$ denotes the restriction of $\pi$ to the set $R$). We removed the monochromatic elements of $\mathcal{F}$ and those that contain $R$, hence $A\cap R,B\cap R \notin \{\emptyset,R\}$. Hence, defining $\delta_A(B)$ and applying the shadow argument as in the proof of Theorem \ref{spernersupersat}, we get that the probability that $\pi(\mathcal{B}_R,\mathcal{B}_W)$ contains $(A,B)$ is at most $\max\{2/|R|,2/|W|\}\leq 2/n^{3/4}$, as claimed. Putting the last two paragraphs together, we obtain that there are at least $\eps \binom{n}{n/2}\frac{n^{3/4}}{4}$ bad pairs, as required.
\end{proof}

Now a simple application of the container method, exactly as in the proof of Theorem \ref{numberofsperner}, yields a counting version of Theorem \ref{generalisedspernerthm}.  

\begin{thm}\label{twocolourspernercounting}
The number of families $\mathcal{F}$ for which there exists a colouring $R\cup W=[n]$ such that there is no comparable pair $A,B\in\mathcal{F}$ with monochromatic difference is $2^{\binom{n}{n/2}(1+o(1))}$.
\end{thm}

\begin{proof}
Fix a colouring $R,W$ with $R\cup W=[n]$, and define a graph $G$ on vertex set $\mathcal{P}(n)$ where two comparable sets are adjacent if their difference is monochromatic. Hence families without comparable pairs with monochromatic difference  correspond to independent sets in $G$. Arguing precisely as in the proof of Theorem \ref{numberofsperner}, we find that the number of independent sets, and hence the number of families without comparable sets with monochromatic difference, is $2^{\binom{n}{n/2}\left(1+o(1)\right)}$. There are $2^n$ possible colourings to start with, hence the number of families for which there exists a colouring avoiding comparable sets with monochromatic difference is $2^n\cdot 2^{\binom{n}{n/2}\left(1+o(1)\right)}=2^{\binom{n}{n/2}\left(1+o(1)\right)}$ as required.
\end{proof}

\subsection{Bollob\'{a}s's inequality and open questions}\label{sec73}
We next consider a generalisation of Sperner's theorem due to Bollob\'{a}s~\cite{bollobassperner}.
We say that a family $\mathcal{F}\subseteq\mathcal{P}(N)\times\mathcal{P}(N)$ is an \emph{intersecting set-pair system (ISP-system) with parameters $(n,N)$} if $\mathcal{F}=\{(A_j,B_j): 1\leq j\leq m\}$ and the following hold:
\begin{enumerate}
\item[(1)] $A_j\cap B_j=\emptyset$ for all $j$,
\item[(2)] $A_j\cap B_k\neq \emptyset$ for all $j\neq k$, and
\item[(3)]  $|A_j|+|B_j|\leq n$ for all $j$.
\end{enumerate}

\begin{thm}\label{griggsstahltrotter}
If $\mathcal{F}$ is an ISP-system with parameters $(n,N)$ then $|\mathcal{F}|\leq \binom{n}{n/2}$ and this bound is attained whenever $N\geq n$.
\end{thm}

We wish to prove that the number of ISP-systems with parameters $(n,N)$ is $2^{(1+o(1))\binom{n}{n/2}}$, where the $o(1)$ term goes to zero as $n$ tends to infinity. However, if $N$ is too large then we cannot possibly hope for such a theorem, hence we need an upper bound for $N$ for this to be true. 

The first step towards a counting theorem is a supersaturation result - suppose $|\mathcal{F}|=\binom{n}{n/2}+x$, $\mathcal{F}$ satisfies $(1)$ and $(3)$, and prove that there are many pairs $(i,j)$ such that the pairs $(A_i,B_i),(A_j,B_j)$ do not satisfy $(2)$. Following the standard proof of Theorem \ref{griggsstahltrotter}, one might want to define the \emph{ball around }$(A,B)$ as the set of permutations of $[N]$ where all the elements of $A$ come before all the elements of $B$. Then, if we could show that the intersection of any two balls is much smaller than the volume of the ball itself, supersaturation would be immediate (see Lemma \ref{supersatforhammingcodes}). Unfortunately, the intersection of two balls need not be very small - and in fact we were not able to establish supersaturation in this case.

\begin{conj}[\label{falseconjbollobas1}Supersaturation for ISP-systems]
Suppose $\mathcal{F}\subset \mathcal{P}(N)\times \mathcal{P}(N)$ of size $|\mathcal{F}|=\binom{n}{n/2}+x$, satisfying $(1)$ and $(3)$ above. Then there are at least $\left(\lfloor\frac{n}{2}\rfloor+1\right)x$ pairs $(A_i,B_i),(A_j,B_j)\in\mathcal{F}$ not satisfying $(2)$. 
\end{conj}

There are two notable special cases of the above conjecture. The first is when we add the constraint $|A_i|=|B_i|=n/2$ for all $i$. The second is when we require $n=N$, that is, $A_i,B_i\subseteq [n]$ for all $i$. Quite surprisingly, we were not even able to establish either one of these special cases of the conjecture!

We note that once the above (or a similar) conjecture is established, a counting theorem bounding the number of ISP-systems follows from a straightforward application of the container method.

~

\emph{Note added in 2018: Conjecture~\ref{falseconjbollobas1} and both special cases mentioned above are false.}

Modifying Lov\'{a}sz's proof \cite{lovaszbollobassperner} of Bollob\'{a}s's inequality, Frankl \cite{franklbollobassperner} and later Kalai \cite{kalaibollobassperner} obtained a stronger (`skew'-)version of Theorem \ref{griggsstahltrotter}. To finish off this section we will try to attack this generalisation. This is only to illustrate that supersaturation for such theorems does \emph{not} always hold.

\begin{thm}[Skew-Bollob\'{a}s inequality]\label{skewbollobas}
Let $\mathcal{F}=\{(A_1,B_1),\ldots,(A_m,B_m)\}$ be an ordered family of ordered pairs of sets satisfying the following:
\begin{enumerate}
\item[(i)] $|A_i|\leq a$ and $|B_i|\leq b$ for $1\leq i \leq m$;
\item[(ii)] $A_i\cap B_i=\emptyset$ for $1\leq i \leq m$;
\item[(iii)] $1\leq i < j \leq m $ implies $A_i\cap B_j\neq \emptyset$.
\end{enumerate}
Then $m\leq \binom{a+b}{a}$.
\end{thm}

We now show a construction that shows that supersaturation does not occur in the case of the skew-Bollob\'{a}s inequality.
\begin{constru}\label{nosupersatconstru}
 \emph{Let $n$ be even and let $a=b=n/2$. Define $$\mathcal{F}_1:=\{(A,B):A\subset \{2,3,\ldots,n\}, |A|=n/2, B=\{2,3,\ldots,n+1\}\backslash A\}$$ and $$\mathcal{F}_2:=\{(A,B): A\subset [n], |A|=n/2, B=[n]\backslash A\}.$$ Let $\mathcal{F}$ be the ordered family starting with all elements of $\mathcal{F}_1$ and then all elements of $\mathcal{F}_2$. Then $|\mathcal{F}|=\binom{n}{n/2}+\binom{n-1}{n/2}=\frac{3}{2}\binom{n}{n/2}$, yet $\mathcal{F}$ only contains $\binom{n-1}{n/2}=\frac{1}{2}\binom{n}{n/2}$ pairs that do not satisfy $(iii)$. Hence attempting to obtain a counting version of Theorem \ref{skewbollobas} using a straightforward application of the container method is not possible: indeed, suppose we aim for a collection of containers of size $(1+\eps)\binom{n}{n/2}$. Suppose for simplicity (and this would make our life }much\emph{ easier!) we are satisfied with counting families that are contained in $\mathcal{F}_1\cup \mathcal{F}_2$.  Then by Construction~\ref{nosupersatconstru} the biggest $\Delta$ we could possibly take is 1! Unfortunately the size of the set $S$ is then roughly $\binom{n}{n/2}/2$ and hence the number of containers will be essentially the same as the trivial bound , i.e. $$|\mathcal{C}|\approx c\binom{\frac32\binom{n}{n/2}}{\binom{n}{n/2}}\approx 2^{1.38\binom{n}{n/2}}\gg 2^{(1+o(1))\binom{n}{n/2}}.$$Hence the straightforward way of applying the method does not yield good results.}
\end{constru}

\medskip


Throughout this paper we have seen several instances of the following theme:

\vspace{.1in}
\noindent\mybox{gray}{
\vspace{.2in}
\centerline{\emph{The presence of supersaturation usually implies strong counting theorems.}}
\vspace{.2in}
}
\vspace{.1in}

It is natural to ask about the converse:

\vspace{.1in}
\noindent\mybox{gray}{
\vspace{.2in}
\centerline{\emph{Does the absence of supersaturation imply in general that we cannot}}
\centerline{\emph{hope for strong counting theorems?}}
\vspace{.2in}
}
\vspace{.1in}

Due to insufficient evidence we do not dare to say that this is true in general. There are, however, many instances in which the answer is affirmative - we direct the reader to the several papers on enumerating $H$-free graphs for a fixed bipartite $H$ \cite{bs2,kw2,morrissaxtoncycle}. The skew-Bollob\'{a}s Theorem (Theorem \ref{skewbollobas}) is another example where the answer to this question is positive. 

\begin{observation}
One might hope that the number of families in $\mathcal{P}(n+1)\times \mathcal{P}(n+1)$ that satisfy (i)--(iii) of Theorem \ref{skewbollobas} (with $a=b=n/2$)  is $2^{(1+o(1))\binom{n}{n/2}}$. Failure of supersaturation is a bad sign - and indeed, the number of such families is at least $2^{1.29\binom{n}{n/2}}$.
\end{observation}
\begin{proof}
Consider the families $\mathcal{F}_1,\mathcal{F}_2$ defined in Construction \ref{nosupersatconstru}. Write $\mathcal{F}_2=\mathcal{F}_2'\cup \mathcal{F}_2''$, where $\mathcal{F}_2'=\{(A,B): A\subset \{2,3,\ldots,n\}, |A|=n/2, B=[n]\backslash A\}$ and  $\mathcal{F}_2''=\{(A,B): A\subset [n], 1\in A, |A|=n/2, B=[n]\backslash A\}$. Note that for each $A\subseteq\{2,3,\ldots,n\}$, there is precisely one pair $(B_1,B_2)$ such that $(A,B_1)\in \mathcal{F}_1$ and $(A,B_2)\in\mathcal{F}_2'$. Hence the elements of $\mathcal{F}_1$ and $\mathcal{F}_2'$ can be paired up with each other: two pairs get paired up if their first sets are the same - call these the \emph{bad pairs}.

Consider any ordered  subfamily $\mathcal{G}$ of $\mathcal F$ 
 that consists of any subfamily of $\mathcal{F}_2''$, plus any subfamily of $\mathcal{F}_1\cup \mathcal{F}_2'$ containing at most one element of each bad pair (and where the  ordering of $\mathcal G$ is inherited by the ordering of $\mathcal F$).
Then $\mathcal G$
satisfies (i)--(iii) from Theorem \ref{skewbollobas}. The number of ways to choose the subfamily of $\mathcal{F}_2''$ is $2^{\frac{1}{2}\binom{n}{n/2}}$, and the number of ways to choose an appropriate subfamily of $\mathcal{F}_1\cup \mathcal{F}_2'$ is $3^{\frac{1}{2}\binom{n}{n/2}}$, and the result follows.
\end{proof}
What about the number of families in $\mathcal{P}(n)\times \mathcal{P}(n)$ if we let $a,b\leq n/2$?  Are there any natural constraints to the above theorem under which supersaturation \emph{does} happen? There are lots of open questions in this area to explore.

\section{Counting maximal independent sets and antichains in the Boolean lattice}\label{newestsec}
Most of this paper dealt with finding $\alpha(n)$, the number of families in $\P(n)$ satisfying some given property. We applied different variations of the container method to obtain asymptotics for $\log \alpha(n)$. The reader might be curious whether it is possible to obtain precise asymptotics for $\alpha(n)$ using these (or different) methods. In general though this seems to be a much more difficult task. 
For example, in the problem we consider below, it is difficult to even make a firm guess on the asymptotics.  

For a graph $G$, we say an independent set $I$ of $G$ is  \emph{maximal} if for any $v\in V(G)\setminus I$, we have that $I\cup\{v\}$ is not independent. Let $\mis(G)$ denote the number of maximal independent sets in $G$. Most of the problems discussed below have their origins in \cite{duffusfranklrodl}. Let $\B_{n,k}$ be the graph on vertex set $\binom{[n]}{k}\cup\binom{[n]}{k+1}$, and edges given by inclusion. 
Ilinca and Kahn~\cite{IK} proved that  $\log_2 \mis (\B_{n,k})=(1+o(1))\binom{n-1}{k}$. They also made the following sharp conjecture.
\begin{conj}[\label{ilincakahnconj}Ilinca and Kahn~\cite{IK}]
$$\mis (\B _{n,k})=(1+o(1))n2^{\binom{n-1}{k}},$$
where the $o(1)$ term goes to $0$ as $n\rightarrow \infty$.
\end{conj} 
 The natural lower bound in Conjecture~\ref{ilincakahnconj} follows by defining for each $i\in [n]$ an induced matching $M_i$ in $\B_{n,k}$ of size $\binom{n-1}{k}$
where the edges of the matching are of the form $(B,B\cup{i})$ for $B \in \binom{[n]\setminus \{i\}}{k}$. 
Each of the $2^{|M_i|}$ sets containing precisely one vertex from each edge in $M_i$ extends to a maximal independent set, and each extension is different. This produces
$2^{\binom{n-1}{k}}$ distinct maximal independent sets. By considering each $M_i$ for $i\in[n]$ we obtain a list of $n2^{\binom{n-1}{k}}$ maximal independent sets, containing not too many repetitions.

It turns out however, that this construction can be tweaked to obtain significantly more maximal independent sets.

\begin{prop}\label{kahncounterlemma}
If $|k-n/2|\leq\sqrt{n}$ then
$$\mis(\B_{n,k})=\Omega\left(n^{3/2}2^{\binom{n-1}{k}}\right).$$
\end{prop}
\noindent\emph{Proof strategy:}
We say a triple $\T=(B,r,s)$ is \emph{good} if  $r,s \in [n]$,~ $B \in \binom{[n]}{k}$,~ $1,r\notin B$,~ $r\neq 1$ and $s\in B$.
For each good triple $\T$ we will construct a collection $f(\T)$ of independent sets in $\B_{n,k}$ with $$|f(\T)|=2^{\binom{n-1}{k}-n+1}.$$
For each good triple $\T$ we will extend all elements of $f(\T)$ to a maximal independent set in an arbitrary way. We will show that every maximal independent set in $\B_{n,k}$ is obtained at most twice in this way. The number of good triples is $\Omega\left(\binom{n-1}{k}n^2\right)$, hence a simple calculation will complete the proof.
\begin{proof}
For a set $C$ with $i\notin C$ and $j\in C$, let $C^i:=C\cup \{i\}$ and $C_j:=C\setminus \{j\}$. We define, for example, $C^{i,k}_j$ analogously (assuming $i,k\notin C$ and $j\in C$). Let $M$ be the induced matching in $\B_{n,k}$ of size $\binom{n-1}{k-1}$ given by $M:=\{(C,C^1): 1\notin C, ~ |C|=k\}$. 

Given a good triple $\T=(B,r,s)$, let $U(\T):=\{C: |C|=k+1, B^1_s\subset C\}$ and $D(\T):=\{C: |C|=k, C\subset B^r\}$. Let $e:=(B,B^1)$ and $f:=(B^r_s,B^{1,r}_s)$; notice these are two edges of $M$. Note that $|U(\T)|=n-k$ and $|D(\T)|=k+1$. Every vertex of $D(\T)\cup U(\T)$ is incident to precisely one edge of $M$. Moreover the only two edges that are simultaneously incident to one vertex in $D(\T)$ and one vertex in $U(\T)$ are $e$ and $f$. The collection of independent sets $f(\T)$ is defined as follows.
\begin{itemize}
\item Let $B^r$ and $B^1_s$ be elements of the independent set.
\item\label{seconditem} If an edge of $M$ is not incident to any vertex in $D(\T)\cup U(\T)$ then put exactly one endpoint of that edge into the independent set.
\item\label{thirditem} If an edge of $M$ is incident to precisely one vertex of $D(\T)\cup U(\T)$, choose the other vertex of this edge.
\end{itemize}
Note that this gives us $\binom{n-1}{k-1}-(n-k)-k+1$ free choices, hence $|f(\T)|=2^{\binom{n-1}{k}-n+1}$ as claimed. Every such set constructed is independent. Indeed, this follows since $M$ is an induced matching, there is no edge between $B^r$ and $B^1_s$, and there is no edge between $B^r$ or $B^1_s$ and another vertex in the set. Let $\F$ be the union of the $f(\T)$s over all good triples.

Every such constructed independent set contains precisely one vertex from each edge in $M$ except for precisely two edges ($e,f$ from above). These two edges lie in a unique $6$-cycle in $\B_{n,k}$ ($e,f$ together with $B^1_s$ and $B^r$), hence given any $I\in \F$, there are precisely two good triples giving rise to this $I$. 
Specifically, if $I\in \mathcal F$ arises from a good triple $\T=(B,r,s)$, the only other good triple that `produces' $I$ is $\T '=(B^r_s,s,r)$.

Moreover, if $I,I' \in \F$ where $I \not =I'$ then $I$ and $I'$ lie in different maximal independent sets in $\B _{n,k}$.
Hence a maximal independent set of $\B_{n,k}$ is counted twice by $|\F|$ if it intersects $M$ in $|M|-2$ vertices, and not counted otherwise.

The number of good triples is $$\binom{n-1}{k}k(n-k-1)=\Omega\left(\frac{2^n}{\sqrt{n}}n^2\right) = \Omega\left(n^{3/2}2^n\right),$$
and the result follows.

\end{proof}

In the above argument we started with a good triple and modified the independent sets from the Ilinca--Kahn construction along a $6$-cycle determined by the triple. But we can get a better lower bound by starting out with a collection $S$ of $N>1$ good triples, as long as the sets in the triples are sufficiently far apart (Hamming distance at least $20$, say) so that the modifications do not interfere with each other. Each maximal independent set is then counted at most $2^{N}$ times, and as long as $N$ is not too large the number of choices for the $N$ triples is at least $$\left(\frac{n}{10}\right)^{2N}\binom{\binom{n-1}{k-1}}{N}\geq \left(\frac{n}{10}\right)^{2N}\left(\frac{2^n}{C_1N\sqrt{n}}\right)^{N} \geq 2^{nN} \left(n^{3/2} N^{-1} C_2 \right)^N, $$ for some absolute constants $C_1,C_2>0$. Each good triple decreases the number of free choices on edges of $M$ by at most $n$, hence costing us a factor of $2^n$. So by setting $N=n^{3/2}C_2/2$ we conclude the following result, which is an exponential improvement over Proposition~\ref{kahncounterlemma}:
\begin{prop}\label{kahncounterlemma2}
There exists an absolute constant $C>0$ such that whenever $k,n$ are such that $|k-n/2|\leq \sqrt{n}$ then $$\mis(\B_{n,k})\geq 2^{\binom{n-1}{k-1}+Cn^{3/2}}.$$\qed
\end{prop}
We do not have any reason to believe that Proposition~\ref{kahncounterlemma2} gives the correct order of magnitude of $\mis(\B_{n,k})$; It would be extremely interesting to determine this. However, we suspect this may   be very difficult.

Finally we note that the above result also disproves another conjecture from \cite{IK}.  Write $\text{ma}(\P(n))$ for the number of maximal antichains in $\P(n)$. Ilinca and Kahn \cite{IK} conjectured that $\text{ma}(\P(n))=\Theta \left(n2^{\binom{n-1}{\lfloor n/2 \rfloor}}\right)$. However, since $\text{ma}(\P(n))\geq \mis(\B_{n,k})$ for all $k$, Proposition~\ref{kahncounterlemma2} disproves this conjecture.

\section*{Acknowledgements}
The authors are grateful to the BRIDGE strategic alliance between the University of Birmingham and the University of Illinois at Urbana-Champaign. This research was conducted as part of the `Building Bridges in Mathematics' BRIDGE Seed Fund project. The authors also thank the referees for their careful reviews.

\end{document}